\documentclass[11pt, twoside, a4paper,reqno]{amsart}
\usepackage{amsfonts, amssymb, amscd, amsmath, amsthm}
\usepackage{setspace}
\usepackage[utf8]{inputenc}
\usepackage{graphicx}
 \usepackage{float}
\allowdisplaybreaks 
\usepackage{mathtools}
\usepackage{multicol}
\newcommand\up[1]{\mbox{\raisebox{1pt}{\ensuremath{#1}}}} 

\newcommand\restr[2]{{
  \left.\kern-\nulldelimiterspace 
  #1 
  \littletaller 
  \right|_{#2} 
  }}
\usepackage{amsmath, amsthm, amscd, amsfonts, amssymb, graphicx, color}
\usepackage{float}
\usepackage{tikz}
\usepackage{xcolor}
\usepackage{pgfplots}
\pgfplotsset{compat=1.18, width=8cm}

\allowdisplaybreaks 
\usepackage{mathtools}
\usepackage[bookmarksnumbered, colorlinks, plainpages]{hyperref}

\usepackage{tikz-cd}
\usepackage[T1]{fontenc}%
\usepackage{blindtext}

\newtheorem{thm}{Theorem}[section]
\newtheorem{lemma}[thm]{Lemma}
\newtheorem{prop}[thm]{Proposition}

\theoremstyle{definition}
\newtheorem{defn}[thm]{Definition}
\newtheorem{eg}[thm]{Example}

\theoremstyle{definition}
\newtheorem{rmk}[thm]{Remark}
\theoremstyle{definition}
\newtheorem{note}[thm]{Note}
\theoremstyle{definition}

\numberwithin{equation}{section}

\usepackage[charter]{mathdesign}

\title{Character space and Gelfand type representation of locally $C^{\ast}$-algebra}
\author[Pamula]{Santhosh Kumar Pamula}
\address{Santhosh Kumar Pamula, Department of Mathematical Sciences, Indian Institute of Science  Education and Research (IISER) Mohali, Knowledge City, S.A.S Nagar, Punjab 140306, India.}
\email{santhoshkp@iisermohali.ac.in}
\author[Siddique]{Rifat Siddique}
\address{Rifat Siddique, Department of Mathematical Sciences, Indian Institute of Science  Education and Research (IISER) Mohali, Knowledge City, S.A.S Nagar, Punjab 140306, India.}
\email{mp21008@iisermohali.ac.in}
\subjclass[2020]{46A03, 47B37, 47A60}
\keywords{Character space, Functional calculus, Inductive limit, Locally $C^{\ast}$-algebra, Locally bounded operator, Projective limit.}
\date{\today}

\newcommand{\littletaller}{\mathchoice{\vphantom{\big|}}{}{}{}}
\begin{document}

\maketitle
\begin{abstract}
    In this article, we identify a suitable approach to define the character space of  a commutative unital locally $C^{\ast}$-algebra via the notion of the inductive limit of topological spaces. Also we discuss topological properties of the character space. We establish the Gelfand type representation between a commutative unital locally $C^{\ast}$-algebra and the space of all continuous functions defined on its character space. Equivalently, we prove that every commutative unital locally $C^{\ast}$-algebra is identified with the locally $C^{\ast}$-algebra of continuous functions on its character space through the coherent representation of projective limit of $C^{\ast}$-algebras. Finally, we construct a unital locally $C^{\ast}$-algebra generated by a given locally bounded normal operator and show that its character space is homeomorphic to the local spectrum. Further, we define the functional calculus and prove spectral mapping theorem in this framework. 
    
\end{abstract}
\section{Introduction and Preliminaries}\label{section:1}
In 1971, A. Inoue introduced the notion of locally $C^{\ast}$-algebras \cite{Inoue}. In the literature, locally $C^{\ast}$-algebras are also referred to as LMC$^{\ast}$-algebras \cite{schmudgen},\;b$^\ast$-algebras \cite{AllanG},\;pro C$^\ast$-algebras \cite{Phillips,voiculescu} and multinormed C$^{\ast}$-algebras \cite{dosiev}. However in the simplest way, a locally $C^{\ast}$-algebra can be seen as the projective limit of projective system of $C^{\ast}$-algebras \cite{gheondea}.

In case of commutative unital $C^{\ast}$-algebras, the well-known Gelfand representation gives a complete characterization in terms of space of continuous functions on a compact Hausdorff space. In fact, the compact Hausdorff space is given by space of all characters or multiplicative linear functionals (\cite{kehezhu}). Our aim in this article is to bring out a suitable notion of character space of a commutative unital locally $C^{\ast}$-algebra, establish a Gelfand type representation and define continuous functional calculus at a given locally bounded normal operator. The authors of \cite{arens,George,schmudgen} have studied the similar aspects upto some extent. However, in this work, we give a complete characterization.

We organize this article into three sections. In the first section, we recall important aspects like projective limit, inductive limit of locally convex spaces \cite{gheondea,kothe,narici}. This helps in understanding definitions of locally $C^{\ast}$-algebras and locally Hilbert spaces respectively. In Remark \ref{rmk: Invlimit} and the discussion that followed will reiterate the fact that every locally $C^{\ast}$-algebra is a projective limit of $C^{\ast}$-algebras. Both Example \ref{Example: loc C* alg} and Example \ref{example: loc hilbert eg} give clear description of the concepts defined in this section. Further, we recall the class of locally bounded operators on a given locally Hilbert space or quantized domain following notations given in \cite{SanthoshKP,dosiev,gheondea}.

In section \ref{section: loc maximal ideal sp}, we give a brief discussion on search for a suitable definition of character space of commutative unital locally $C^{\ast}$-algebra $\mathcal{A}.$ Firstly, we show that a multiplicative linear functional on $\mathcal{A}$ is not necessarily continuous (see Example \ref{eg: Phi is not continuous}). In Example \ref{eg: Phi not associated to varphi1}, we show that a multiplicative linear functional may not also be induced from corresponding multiplicative linear functional on every quotient $C^{\ast}$-algebra. Using these observations, the character space is defined in Definition \ref{defn: characterspace} and we prove that it is an inductive limit of inductive system of compact Hausdorff spaces (see Equation \eqref{Eq: M_A inductive limit}). Moreover, the character space is completely regular with respect to the inductive limit topology (defined in Equation \eqref{eq: inductive limit topology}) and shown in Theorem \ref{thm: correspondence} that there is one to one correspondence with certain class of maximal ideals in $\mathcal{A}.$

In the last section, we define Gelfand type representation from $\mathcal{A}$ to the space of continuous functions defined on its character space. In fact, this is a coherent representation (see Theorem \ref{Thm: gelfand coherent}). Next, we define \textit{local spectrum} of a locally bounded operator and through Example \ref{eg: number matrix}, we point out an important observation that local spectrum is different from unbounded spectrum (the well-known spectrum of densely defined unbounded operator). Finally, we introduce the notion of unital locally $C^{\ast}$-algebra generated by a locally bounded normal operator and by using this we define continuous functional calculus, via local spectrum, of the locally bounded normal operator (see Definition \ref{defn: cont func calculus} and Theorem \ref{Thm: cont functional calculus properties}). We conclude this section by proving the (local) spectral mapping theorem.
 \begin{defn}\label{defn: proj system} \cite{gheondea}
     A pair $\big(\{{\mathcal{V}}_{\alpha}\}_{\alpha \in \Lambda},\{\phi_{\alpha,\beta}\}_{\alpha\leq\beta} \big)$ is said to be a projective (or inverse) system of locally convex spaces, if the following conditions are satisfied: 
     \begin{enumerate}
         \item ${(\Lambda, \leq) \text{ is a directed POSET}},$ 
         \item $\{{\mathcal{V}}_{\alpha}\}_{\alpha \in \Lambda}$ is a net of locally convex spaces,
         \item Whenever $\alpha \leq \beta$, the map $\phi_{\alpha, \beta} \colon {\mathcal{V}}_\beta \rightarrow {\mathcal{V}}_\alpha $ is a continuous surjective linear map such that $\phi_{\alpha, \alpha}$ is the identity map on   ${\mathcal{V}_\alpha} \text{ for all } \alpha \in \Lambda,$
         \item The following transitive property holds true: 
         \begin{equation*}
             {\phi}_{\alpha, \gamma} = {\phi}_{\alpha, \beta} \circ {\phi}_{\beta, \gamma}\;, \;\text{whenever}\; \alpha \leq \beta \leq \gamma.
         \end{equation*}
     \end{enumerate}
     Here the condition $(4)$ implies that the following commuting diagram holds true: 
         \begin{equation*}
    \begin{tikzcd}[every matrix/.append style={name=m},
  execute at end picture={
        \draw [<-] ([xshift=0mm,yshift=2mm]m-2-2.north) arc[start angle=-90,delta angle=270,radius=0.25cm];
  }]
    \mathcal{V}_{\beta} \arrow{rr}{\phi_{\alpha, \beta}} && \mathcal{V}_{\alpha} \\
    & \mathcal{V}_{\gamma} \arrow{lu}{\phi_{\beta,{\gamma}}} \arrow{ru}[swap]{\phi_{\alpha, \gamma}}
\end{tikzcd}
\end{equation*}
 \end{defn}
 \noindent A systematic construction of the ``projective limit'' is described in Section 1.1 of \cite{gheondea}. For the sake of completion, we give few details here. 
Suppose that $\big(\{{\mathcal{V}}_{\alpha}\}_{\alpha \in \Lambda},\{\phi_{\alpha,\beta}\}_{\alpha\leq\beta}\big)$ is a projective system of locally convex space, then  consider the vector space 

\begin{equation*}
    \prod\limits_{\alpha \in \Lambda} {\mathcal{V}}_{\alpha} = \big\{\{v_{\alpha}\}_{\alpha \in \Lambda} \colon v_\alpha \in {\mathcal{V}}_{\alpha} , \alpha \in \Lambda \big\},
\end{equation*} 
equipped with the product topology (i.e., the weakest topology on $ \prod\limits_{\alpha \in \Lambda} {\mathcal{V}}_{\alpha} $ such that the projection from $\left(\prod\limits_{\alpha \in \Lambda} {\mathcal{V}}_{\alpha}\right)$ to $ {\mathcal{V}}_\alpha$ is continuous for all $ \alpha \in \Lambda). $ The subspace given by 
\begin{equation} \label{Equation: Proj limit}
    \mathcal{V} : = \Big\{ \{v_{\alpha}\}_{\alpha \in \Lambda} \in \prod\limits_{\alpha \in \Lambda} {\mathcal{V}}_{\alpha} : \; \phi_{\alpha, \beta}(v_{\beta}) = v_{\alpha}\; \text{whenever}\; \alpha \leq \beta \Big\} \subseteq \prod\limits_{\alpha \in \Lambda} {\mathcal{V}}_{\alpha} 
\end{equation}
is equipped with the weak topology induced by the family $\{\phi_{\alpha}\}_{\alpha \in \Lambda}$ of maps  from $\mathcal{V}$ to ${\mathcal{V}}_\alpha$ defined by
\begin{equation*}
    \phi_{\alpha} \big(\{v_{\beta}\}_{\beta \in \Lambda}\big) = v_{\alpha}, \text{ for all } \alpha \in \Lambda .
\end{equation*}
Here the weak topology on $\mathcal{V}$ is the smallest topology such that each of these linear maps $\phi_\alpha$ continuous. Then the pair $\left(\mathcal{V}, \{\phi_\alpha\}_{\alpha \in \Lambda}\right)$ is called a {\bf projective limit} (or inverse limit) of locally convex spaces induced by the projective system $\left(\{\mathcal{V}_\alpha\}_{\alpha \in \Lambda}, \{\phi_{\alpha,\beta}\}_{\alpha \leq \beta}\right)$ and it is denoted by
 \begin{equation*}
     \mathcal{V} = \varprojlim_{\alpha \in \Lambda} \mathcal{V}_\alpha.
 \end{equation*}

Firstly, note that $\big(\mathcal{V}, \{\phi_\alpha\}_{\alpha \in \Lambda}\big)$ is \textit{compatible} with the projective system $\big(\{\mathcal{V}_\alpha\}_{\alpha \in \Lambda}, \{\phi_{\alpha,\beta}\}_{\alpha \leq \beta} \big)$, in the sense that
\begin{equation*}
    \phi_\alpha = \phi_{\alpha,\beta} \circ \phi_\beta,\; \text{whenever}\; \alpha \leq \beta.
\end{equation*}
Suppose $\mathcal{W}$ is a locally convex space and there exists a continuous linear map $\psi_\alpha \colon \mathcal{W} \rightarrow \mathcal{V}_\alpha,$ for each $\alpha \in \Lambda,$ such that the pair $\big(\mathcal{W},\{\psi_\alpha\}_{\alpha \in \Lambda}\big)$ is compatible with the projective system $\big(\{\mathcal{V}_\alpha\}_{\alpha \in \Lambda}, \{\phi_{\alpha,\beta}\}_{\alpha \leq \beta}\big).$ That is, $\psi_\alpha = \phi_{\alpha,\beta} \circ \psi_\beta,\; \text{whenever}\; \alpha \leq \beta.$ It follows that, for any $w \in \mathcal{W},$ the net $\{\psi_\alpha(w)\}_{\alpha \in \Lambda} \in \varprojlim\limits_{\alpha \in \Lambda} \mathcal{V}_\alpha.$ Further,
\begin{enumerate}
    \item there is a natural continuous linear map $\psi \colon \mathcal{W} \rightarrow \varprojlim\limits_{\alpha \in \Lambda} \mathcal{V}_\alpha$ defined by 
    \begin{equation*}
        \psi(w)=\big\{\psi_\alpha(w)\big\}_{\alpha \in \Lambda}, \;\text{for every}\; w \in \mathcal{W}
    \end{equation*}
    such that  $\psi_\alpha =\phi_\alpha \circ \psi,$ for each $\alpha \in \Lambda.$ That is, the following diagram commutes:
\begin{equation*}
\begin{tikzcd}[every matrix/.append style={name=m},
  execute at end picture={
        \draw [<-] ([xshift=0mm,yshift=2mm]m-2-2.north) arc[start angle=-90,delta angle=270,radius=0.25cm];
  }]
    \mathcal{W} \arrow{dr}[left]{\psi} \arrow{rr}{\psi_{\alpha}} && \mathcal{V}_{\alpha} \\
    & \varprojlim\limits_{\alpha \in \Lambda} \mathcal{V}_\alpha \arrow{ur}[right]{\phi_{\alpha}}
\end{tikzcd} \newline
\end{equation*}

\item Such a map $\psi$  is unique: if there is any other continuous linear map $\psi' \colon \mathcal{W} \rightarrow \varprojlim\limits_{\alpha \in \Lambda} \mathcal{V}_\alpha$ satisfying $\psi_\alpha =\phi_\alpha \circ \psi',$ for each $\alpha \in \Lambda,$ then 
\begin{equation*}
    \phi_\alpha\big(\psi'(w)\big)=\psi_\alpha(w), \; \text{for every} \; \alpha \in \Lambda,\, w \in \mathcal{W}.
\end{equation*}  
Since each $\phi_{\alpha}$ is a coordinate projection of $\varprojlim\limits_{\alpha \in \Lambda} \mathcal{V}_{\alpha}$ we get $\psi=\psi'.$
\end{enumerate}
This shows that, the projective limit is unique upto compatibility. 
Now we recall the notion of coherent map between projective limits of two projective systems of locally convex spaces. 
\begin{defn}\label{defn: coherent proj}
Let $(\Lambda,\leq)$ be a directed POSET. Suppose ($\mathcal{V}, \{\phi_\alpha\}_{\alpha \in \Lambda}$) and $(\mathcal{W}, \{\psi_\alpha\}_{\alpha \in \Lambda})$ are projective limits of projective systems $\left(\{\mathcal{V}_\alpha\}_{\alpha \in \Lambda}, \{\phi_{\alpha,\beta}\}_{\alpha \leq \beta}\right)$ and $\left(\{\mathcal{W}_\alpha\}_{\alpha \in \Lambda}, \{\psi_{\alpha,\beta}\}_{\alpha \leq \beta}\right)$ of locally convex spaces respectively. A linear map $f \colon \mathcal{V} \rightarrow \mathcal{W}$ is said to be \textit{coherent} if there exists a net $\{f_\alpha\}_{\alpha \in \Lambda}$ of linear maps from $\mathcal{V}_\alpha$ to $\mathcal{W}_\alpha$, such that $\psi_\alpha \circ f = f_\alpha \circ \phi_\alpha$ for all $\alpha \in \Lambda.$ Equivalently,

\begin{equation*}
    \begin{tikzcd}[every matrix/.append style={name=m},
    execute at end picture={
    \draw [<-] ([xshift=7mm,yshift=2mm]m-2-2.west) arc[start angle=-90,delta angle=270,radius=0.33cm]; }] 
    \mathcal{V}  \arrow[thick,swap] {dd}{\phi_\alpha} \arrow[thick]{rr}{f} && \mathcal{W}  \arrow[thick]{dd}{\psi_\alpha} \\
     && \\
    \mathcal{V}_\alpha \arrow[thick,swap]{rr}{f_\alpha} && \mathcal{W}_\alpha   
    \end{tikzcd}
\end{equation*}
\end{defn}
In the beginning of this section, a brief discussion on projective limit of projective system of locally convex spaces is presented. We use this notion to recall the definition of a locally $C^{\ast}$-algebra. 
\begin{defn}\label{defn: loc C*-alg}
Let $\mathcal{A}$ be a unital complex $\ast$-algebra. A seminorm $p$ on $\mathcal{A}$ is said to be a $C^{\ast}$-seminorm, if 
\begin{multicols}{2}
    \begin{enumerate}
        \item[(a)] $p(1_\mathcal{A}) = 1$\\
        \item[(b)] $p(ab) \leq p(a) p(b)$
        \item[(c)] $p(a^\ast) = p(a)$\\
        \item[(d)] $p(a^*a) = p(a)^2,$
    \end{enumerate}
\end{multicols}
\noindent for all $a,b \in \mathcal{A}$. Suppose that $(\Lambda, \leq)$ is a directed POSET, then 
\begin{enumerate}
    \item a family $\{p_{\alpha}\}_{\alpha \in \Lambda}$ of $C^{\ast}$-seminorms is called \textit{upward filtered}, if  for every $a \in \mathcal{A}$, we have $p_\alpha(a) \leq p_\beta(a),$ whenever $\alpha \leq \beta$ ;
    \item the unital $\ast$-algebra $\mathcal{A}$ is said to be a \textit{locally $C^\ast$-algebra} if it has a complete Hausdorff locally convex topology induced by an upward filtered family $\{p_\alpha\}_{\alpha\in \Lambda}$ of $C^{\ast}$-seminorms defined on $\mathcal{A}$.
\end{enumerate}
\end{defn}
Note that, in the literature, \textit{locally $C^{\ast}$-algebras} are also known as \textit{$LMC^{\ast}$-algebras} in \cite{schmudgen}, \textit{$b^{\ast}$-algebras} in \cite{AllanG}, \textit{pro $C^{\ast}$-algebras} in \cite{Phillips,voiculescu} and \textit{multinormed $C^{\ast}$-algebras} in \cite{dosiev}. In the following remark we construct locally $C^{\ast}$-algebra from a certain family of $C^{\ast}$-algebras.
\begin{rmk} \label{rmk: Invlimit}
    Let $(\Lambda, \leq)$ be a directed POSET and $\Big( \{\mathcal{B}_{\alpha}\}_{\alpha \in \Lambda},\; \{\phi_{\alpha, \beta}\}_{\alpha \leq \beta}\Big)$ be a projective system of unital $C^{\ast}$-algebras. That is, whenever $\alpha \leq \beta$ the map $\phi_{\alpha, \beta} \colon \mathcal{B}_{\beta} \to \mathcal{B}_{\alpha}$ is a unital surjective $C^{\ast}$-homomorphism,  satisfying the properties listed in Definition \ref{defn: proj system}. Following the Equation \eqref{Equation: Proj limit} we consider the projective limit,
    \begin{equation} \label{Equation: Proj limit C*}
        \varprojlim\limits_{\alpha \in \Lambda} \mathcal{B}_{\alpha} = \Big\{ \{x_{\alpha}\}_{\alpha \in \Lambda} \in \prod\limits_{\alpha \in \Lambda} \mathcal{B}_{\alpha}:\; \phi_{\alpha, \beta}(x_{\beta}) = x_{\alpha},\; \text{whenever}\; \alpha \leq \beta \Big\}.
    \end{equation}
    It is clear from earlier discussion that it is a linear space. Let $\{x_{\alpha}\}_{\alpha \in \Lambda},\; \{y_{\alpha}\}_{\alpha \in \Lambda} \in \varprojlim\limits_{\alpha \in \Lambda} \mathcal{B}_{\alpha}$, we define 
    \begin{equation*}
        \{x_{\alpha}\}_{\alpha \in \Lambda} \cdot \{y_{\alpha}\}_{\alpha \in \Lambda} = \{ x_{\alpha} y_{\alpha}\}_{\alpha \in \Lambda} \; \text{and} \; \{x_{\alpha}\}^{\ast}_{\alpha \in \Lambda} = \{x_{\alpha}^{\ast}\}_{\alpha \in \Lambda}.
    \end{equation*}
    Since $\phi_{\alpha, \beta} (x_{\beta}y_{\beta}) = \phi_{\alpha, \beta}(x_{\beta}) \phi_{\alpha, \beta}(y_{\beta}) = x_{\alpha}y_{\alpha} $ and $\phi_{\alpha, \beta}(x_{\beta}^{\ast}) = \phi_{\alpha, \beta}(x_{\beta})^{\ast} = x_{\alpha}^{\ast}$ whenever $\alpha \leq \beta$, it follows that $\varprojlim\limits_{\alpha \in \Lambda} \mathcal{B}_{\alpha}$ is a unital $\ast$-algebra. Suppose that for each $ \beta \in \Lambda$, if we define the semi-norm on $\varprojlim\limits_{\alpha \in \Lambda} \mathcal{B}_{\alpha}$ as,
    \begin{equation*}
        q_{\beta}\big(\{x_{\alpha}\}_{\alpha \in \Lambda}\big) : = \| x_{\beta} \|_{\mathcal{B}_{\beta}}
    \end{equation*}
    then by using the fact that $\phi_{\alpha, \beta}$ is contractive (whenever $\alpha \leq \beta$) we see that $\{q_{\alpha}\}_{\alpha \in \Lambda}$ is an upward filtered family of $C^{\ast}$-seminorms. Further, it follows that the locally convex topology induced by this family $\{q_{\alpha}\}_{\alpha \in \Lambda}$ is the weakest topology such that the map $\phi_{\alpha}\colon \varprojlim\limits_{\alpha \in \Lambda} \mathcal{B}_{\alpha} \to \mathcal{B}_{\alpha}$ given by $\{x_{\alpha}\}_{\alpha \in \Lambda} \mapsto x_{\alpha}$ continuous for every $\alpha \in \Lambda$.  Since each $\mathcal{B}_{\alpha}$ is complete, from \cite[Section 1.1]{gheondea} we conclude that $\varprojlim\limits_{\alpha \in \Lambda} \mathcal{B}_{\alpha}$ is complete with respect to the locally convex topology generated by an upward filtered family $\{q_{\alpha}\}_{\alpha \in \Lambda}$ of $C^{\ast}$-seminorms. Hence $\varprojlim\limits_{\alpha \in \Lambda} \mathcal{B}_{\alpha}$ is a locally $C^{\ast}$-algebra.
\end{rmk}
In the Remark \ref{rmk: Invlimit} we have shown that the projective limit of projective system of $C^{\ast}$-algebras is a locally $C^{\ast}$-algebras. The converse of the statment holds true, that is, every locally $C^{\ast}$-algebra can be seen as a projective limit of a projective system of $C^{\ast}$-algebras. We recall the main construction from \cite{gheondea}. Let $\mathcal{A}$ be a locally $C^{\ast}$-algebra and $\mathcal{P}(\mathcal{A})$ denotes the collection of all continuous \textit{C*}-seminorms on $\mathcal{A}.$ Here continuity is with respect to the Hausdorff locally convex topology mentioned in the above Definition \ref{defn: loc C*-alg}. Then $\mathcal{P}(\mathcal{A})$ is a directed POSET with respect to the partial order "$\leq$" given by 
\begin{equation*}
     p \leq q \text{ in } \mathcal{P}(\mathcal{A}) \text{ iff } p(a) \leq q(a), \text{ for all } a \in \mathcal{A}.
\end{equation*}
Without loss of generality, we consider the upward filtered family $\{p_\alpha\}_{\alpha \in \Lambda}$ of continuous $C^{\ast}$-seminorms, where $(\Lambda, \leq)$ is a directed POSET. 
 For each $\alpha \in \Lambda$, define
 \begin{equation*}
     \mathcal{I}_\alpha = \big\{a \in \mathcal{A} \colon p_\alpha(a) = 0\big\}.
 \end{equation*}
Clearly, $\mathcal{I}_\alpha$ is a closed two sided $\ast$-ideal in $\mathcal{A},$ and $\mathcal{A}_\alpha \colon = \mathcal{A}/\mathcal{I}_\alpha$ is a quotient $C^{\ast}$-algebra with respect to the $C^{\ast}$- norm $\|.\|_\alpha$ given by
\begin{equation*}
\|a+\mathcal{I}_\alpha\|_\alpha = p_\alpha(a), \text{ for all } a \in \mathcal{A}.
\end{equation*}
So, we get a family $\{\mathcal{A}_{\alpha}\}_{\alpha \in \Lambda}$ of $C^{\ast}$-algebras. Whenever $\alpha \leq \beta$, we see that $\mathcal{I}_{\beta} \subseteq \mathcal{I}_{\alpha}$ and there is a natural surjective $C^{\ast}$-homomorphism $\pi_{\alpha, \beta}\colon \mathcal{A}_{\beta} \to \mathcal{A}_{\alpha}$ defined by 
\begin{equation*}
    \pi_{\alpha, \beta}(a+\mathcal{I}_{\beta}) = a+\mathcal{I}_{\alpha},\; \text{for every }\; a \in \mathcal{A}.
\end{equation*}
Then $\big( \{\mathcal{A}_{\alpha}\}_{\alpha \in \Lambda},\; \{\pi_{\alpha, \beta}\}_{\alpha \leq \beta}\big)$ forms a projective system of $C^{\ast}$-algebras. Now, for each $\alpha \in \Lambda$, by considering the canonical quotient map $\pi_{\alpha}\colon \mathcal{A} \to \mathcal{A}_{\alpha}$ we see that the pair $\big( \mathcal{A},\; \{\pi_{\alpha}\}_{\alpha \in \Lambda}\big)$ is compatible with the projective system $\big( \{\mathcal{A}_{\alpha}\}_{\alpha \in \Lambda},\; \{\pi_{\alpha, \beta}\}_{\alpha \leq \beta}\big)$ since $\pi_{\alpha,\beta} \circ \pi_\beta(a)=\pi_\alpha(a), \; \text{whenever}\; \alpha \leq \beta.$

On the other hand, if we define $\varprojlim\limits_{\alpha \in \Lambda} \mathcal{A}_{\alpha}$ defined as in Equation \eqref{Equation: Proj limit C*}, then for each $\alpha \in \Lambda$ the map $\psi_{\alpha} \colon \varprojlim\limits_{\alpha \in \Lambda} \mathcal{A}_{\alpha} \to \mathcal{A}_{\alpha}$ given by 
\begin{equation*}
    \psi_{\alpha}\big(\{x_{\alpha}\}_{\alpha}\big) = x_{\alpha}, \; \text{for every}\; \{x_{\alpha}\}_{\alpha \in \Lambda} \in \varprojlim\limits_{\alpha \in \Lambda} \mathcal{A}_{\alpha}
\end{equation*}
is a surjective $\ast$-homomorphism such that for every $\{x_{\alpha}\}_{\alpha \in \Lambda}$, we have
\begin{equation*}
    \big(\pi_{\alpha, \beta} \circ \psi_{\beta}\big) (\{x_{\alpha}\}_{\alpha \in \Lambda}) = \pi_{\alpha, \beta} (x_{\beta}) = x_{\alpha} = \psi_{\alpha}(\{x_{\alpha}\}_{\alpha \in \Lambda}),\; \text{whenever}\; \alpha \leq \beta.
\end{equation*}
It shows that the pair $\big( \varprojlim\limits_{\alpha \in \Lambda} \mathcal{A}_{\alpha},\; \{\psi_{\alpha}\}_{\alpha \in \Lambda}\big)$ is compatible with the projective system $\big( \{\mathcal{A}_{\alpha}\}_{\alpha},\{\pi_{\alpha, \beta}\}_{\alpha \leq \beta}\big)$. Therefore, there exists a unique $\ast$-homomorphism $\psi \colon \mathcal{A} \to \varprojlim\limits_{\alpha \in \Lambda} \mathcal{A}_{\alpha} $ given by $a \mapsto \big\{ a+ \mathcal{I}_{\alpha}\big\}_{\alpha \in \Lambda}$ and satisfying,
\begin{equation*}
    (\psi_{\alpha}\circ \psi)(a) = \psi_{\alpha}(\{a+\mathcal{I}_{\alpha}\}) = a+\mathcal{I}_{\alpha} = \pi_{\alpha}(a),\; \text{for every}\; a \in \mathcal{A},\; \alpha \in \Lambda.
\end{equation*}
By the uniqueness (upto compatibility) of projective limit, we conclude that $\mathcal{A} = \varprojlim\limits_{\alpha \in \Lambda} \mathcal{A}_{\alpha}.$ Hence every locally $C^{\ast}$-algebra can be seen as a projective limit of projective system of $C^{\ast}$-algebras. 
\begin{note}
    In order to show a unital $\ast$-algebra is locally $C^{\ast}$- algebra, from the above discussion, it is enough to show that given $\ast$-algebra is compatible with some projective system of $C^{\ast}$-algebras.
\end{note}
  Now, we describe the notion of locally $C^{\ast}$-algebra with an example below.
\begin{eg}\label{Example: loc C* alg}
  Let $\Lambda = \mathbb{N}$ and let $\mathcal{A}_n = C([-n,n])$ be the unital $C^{\ast}$-algebra of complex valued  continuous functions on the interval $[-n,n]$, (for all $n \in \mathbb{N}$) with respect to supremum norm. For every $m \leq n,$ define $\phi_{m,n} : \mathcal{A}_n \rightarrow \mathcal{A}_m$ by
  \begin{equation*}
      \phi_{m,n}(f) = \restr{f}{[-m,m]}, \text{ for all } f \in \mathcal{A}_n
  \end{equation*}
  Clearly, $\phi_{m,n}$ is a surjective $C^\ast$-homomorphism of $C^{\ast}$-algebras and so $\big(\{\mathcal{A}_{n}\}_{n \in \mathbb{N}}, \{\phi_{m,n}\}_{m \leq n}\big)$ is a projective system.
Now we consider the unital $\ast$-algebra $C(\mathbb{R})$ of continuous functions on $\mathbb{R}$  and for each $n \in \mathbb{N},$ define the seminorm $p_n$ on $C(\mathbb{R})$ by
\begin{equation*}
    p_n(f)=\sup\big\{|f(t)| \colon t \in [-n,n]\big\},\text{ for all } f \in C(\mathbb{R}).
\end{equation*}
Here $p_m(f) \leq p_n(f),$ for all $f \in C(\mathbb{R}),$ whenever $m \leq n.$ Note that $C(\mathbb{R})$ is equipped with the locally convex topology given by the upward filtered family $\{p_n\}_{n \in \mathbb{N}}$ of $C^{\ast}$-seminorms. It follows that the map $\phi_n : C(\mathbb{R}) \rightarrow C([-n,n])$ defined by 
\begin{equation*}
    \phi_n(f) = \restr{f}{[-n,n]}, \text{ for all } f \in C(\mathbb{R}),
\end{equation*}
 is continuous for all $n \in \mathbb{N}.$ For which consider an open set $U$ given by 
 \begin{equation*}
     U=\big\{f \in C([-n,n]) \colon |f(x)|<r, \text{ for all }x \in [-n,n]\big\}
 \end{equation*}in $C([-n,n])$. Then
\begin{align*}
    \phi^{-1}_{n}(U)&=\big\{f \in C(\mathbb{R}) \colon |\phi_n(f)(t)|<r, \text{ for all }t\in[-n,n]\big\}\\
    &=\big\{f \in C(\mathbb{R}) \colon |f(t)|<r, \text{ for all }t\in[-n,n]\big\} \\
    &= p^{-1}_{n}\big(B(0,r)\big),
\end{align*}
where $B(0,r)$ is an open ball in $\mathbb{C}$ of radius $r$ centered at $0.$ Thus, $\phi^{-1}_{n}(U)$ is open in the locally convex topology and consequently $\phi_n$ is continuous for all $n \in \mathbb{N}.$ Further, whenever $m \leq n$, we have  
\begin{align*}
    \phi_{m,n} \circ \phi_n(f)
     = \phi_{m,n}(\restr{f}{[-n,n]})
     = \restr{f}{[-m,m]}
     = \phi_m(f),
\end{align*}
for all $f \in C(\mathbb{R}).$ So, the pair $\big(C(\mathbb{R}),\{\phi_n\}_{n \in \mathbb{N}}\big)$ is compatible with the projective system $\big(\{\mathcal{A}_{n}\}_{n \in \mathbb{N}} , \{\phi_{m,n}\}_{m \leq n}\big)$ of $C^{\ast}$-algebras. Hence  
$$C(\mathbb{R}) = \varprojlim\limits_{n \in \mathbb{N}} C([-n,n]). $$
Equivalently, $C(\mathbb{R})$ is a locally $C^{\ast}$-algebra.
\end{eg}
\begin{note}
     Every $C^{\ast}$-algebra is a locally $C^{\ast}$-algebra with respect to the $C^{\ast}$-norm. The converse may not be true. For instance in Example \ref{Example: loc C* alg} we have seen that $C(\mathbb{R})$ is a locally $C^{\ast}$-algebra but not a $C^{\ast}$-algebra.
\end{note}


\subsection{Inductive limit} In order to understand the notion of locally Hilbert space, we recall from 
\cite[Subsection 1.2]{gheondea} the concept of inductive limit of inductive system of locally convex spaces.  
\begin{defn}\label{defn: ind system}
A pair $\left(\{{\mathcal{X}}_{\alpha}\}_{\alpha \in \Lambda},\{\psi_{\beta, \alpha}\}_{\alpha\leq\beta}\right)$ is called an \textit{inductive system} of locally convex spaces if it satisfies the following conditions:
\begin{enumerate}
    \item[(a)] ${(\Lambda, \leq) \text{ is a directed POSET,}}$
    \item[(b)]  $\{{\mathcal{X}}_{\alpha}\}_{\alpha \in \Lambda}$ is a net of locally convex spaces, 
    \item[(c)]  Whenever $\alpha \leq \beta,$ the map $\psi_{\beta,\alpha}: {\mathcal{X}}_\alpha \rightarrow {\mathcal{X}}_\beta$ is a continuous linear map such that $\psi_{\alpha, \alpha}$ is the identity map on ${\mathcal{X}_\alpha} \text{ for all } \alpha \in \Lambda,$ 
    \item[(d)] The following transitive condition holds true:
\begin{equation*}
  \hspace{8pt} {\psi}_{{\gamma}, \alpha} = {\psi}_{{\gamma}, \beta} \circ {\psi}_{\beta, \alpha}, \text{whenever } \alpha \leq \beta \leq \up{\gamma}.
\end{equation*}
In other words, the following commuting diagram holds true:
\begin{equation*}
\begin{tikzcd}[every matrix/.append style={name=m},
  execute at end picture={
        \draw [<-] ([xshift=0mm,yshift=3mm]m-2-2.north) arc[start angle=-90,delta angle=270,radius=0.25cm];
  }]
    \mathcal{X}_{\alpha} \arrow{dr}[swap]{\psi_{{\gamma},\alpha}} \arrow{rr}{\psi_{\beta,\alpha}} && \mathcal{X}_{\beta} \arrow{dl}{\psi_{{\gamma}, \beta}}\\
    & \mathcal{X}_{{\gamma}}  
\end{tikzcd} \newline
\end{equation*}
\end{enumerate}
\end{defn}
An explicit description of the inductive limit of an inductive system is given below. Firstly, consider the direct sum, 
\begin{equation*}
   \bigoplus\limits_{\alpha \in \Lambda}\mathcal{X}_\alpha := \big\{\{x_\alpha\}_{\alpha \in \Lambda} \in \prod\limits_{\alpha \in \Lambda} \mathcal{X}_\alpha \colon \{x_\alpha\}_{\alpha \in \Lambda} \text{ is of finite support}\big\} \subseteq \prod\limits_{\alpha \in \Lambda} \mathcal{X}_\alpha
\end{equation*}
is endowed with the strongest locally convex topology such that the canonical inclusion maps $\mathcal{X}_\beta \hookrightarrow \bigoplus\limits_{\alpha \in \Lambda}\mathcal{X}_\alpha$ continuous, for all $\beta \in \Lambda.$ For $\xi \in \mathcal{X}_{\beta}$, we denote $\delta_{\xi}$ by the vector in $\bigoplus\limits_{\alpha \in \Lambda} \mathcal{X}_{\alpha}$ that has $\xi$ at $\beta$ and zero elsewhere. It follows that each $\mathcal{X}_{\alpha}$ is canonically identified with a subspace of $\bigoplus\limits_{\alpha \in \Lambda}\mathcal{X}_\alpha$ through the map $\mathcal{X}_\alpha \ni x_\alpha \mapsto \delta_{x_\alpha}.$  Now we consider the linear subspace $\mathcal{X}_0 \text{ of } \bigoplus\limits_{\alpha \in \Lambda}\mathcal{X}_\alpha$ defined by
\begin{equation} \label{equation: X0}
    \mathcal{X}_0 := span\big\{\delta_{x_\alpha} - \delta_{\psi_{\beta,\alpha}(x_\alpha)} \colon \alpha,\beta \in \Lambda, \alpha \leq \beta, x_\alpha \in \mathcal{X}_\alpha\big\}.
\end{equation}
It is immediate to see that $\mathcal{X}_{0}$ is a closed subspace of $ \bigoplus\limits_{\alpha \in \Lambda}\mathcal{X}_\alpha.$ The inductive limit is defined by the following quotient space,
\begin{equation} \label{Equation: indlimit}
    \varinjlim\limits_{\alpha \in \Lambda} \mathcal{X}_{\alpha}: = \Big( \bigoplus\limits_{\alpha \in \Lambda} \mathcal{X}_{\alpha}\Big) \big/ \mathcal{X}_{0}.
\end{equation}
Typically, vectors in $\varinjlim\limits_{\alpha \in \Lambda} \mathcal{X}_{\alpha}$ are cosets of the form $\{x_{\alpha}\}_{\alpha \in \Lambda} + \mathcal{X}_{0}$, where $\{x_{\alpha}\}_{\alpha} \in \bigoplus\limits_{\alpha \in \Lambda} \mathcal{X}_{\alpha}.$ For each $\alpha \in \Lambda$, there is a natural canonical linear map $\psi_\alpha : \mathcal{X}_\alpha \rightarrow \varinjlim\limits_{\alpha\in \Lambda}\mathcal{X}_\alpha$ defined by
\begin{equation*}
    \psi_{\alpha}(x_{\alpha}) = \delta_{x_{\alpha}} + \mathcal{X}_{0},\; \text{for every}\; x_{\alpha} \in \mathcal{X}_{\alpha}.
\end{equation*}
 The  topology on $ \varinjlim\limits_{\alpha\in \Lambda}\mathcal{X}_\alpha $ is the strongest locally convex topology such that the linear maps $\mathcal{\psi}_\alpha$ continuous, for all $\alpha \in \Lambda$. It is known as the inductive limit topology. Further, if $x_{\alpha} \in \mathcal{X}_{\alpha}$ and $\alpha \leq \beta$ then $\delta_{x_{\alpha}} - \delta_{\psi_{\beta, \alpha}(x_{\alpha})} \in \mathcal{X}_{0}$ and so, 
 \begin{align*}
    \psi_{\alpha}(x_{\alpha}) = \delta_{x_{\alpha}} + \mathcal{X}_{0} =  \delta_{\psi_{\beta, \alpha}(x_{\alpha})} + \mathcal{X}_{0}  = \psi_{\beta}\big( \psi_{\beta, \alpha}(x_{\alpha})\big) = \big(\psi_{\beta}\circ \psi_{\beta, \alpha}\big)(x_{\alpha}). 
 \end{align*}
 Since $x_{\alpha}$ is arbitrary in $\mathcal{X}_{\alpha}$, we conclude that 
 \begin{equation} \label{Equation: inductivecompatibility}
     \psi_{\alpha} = \psi_{\beta} \circ \psi_{\beta, \alpha},\; \text{whenever}\; \alpha \leq \beta.
 \end{equation}
 Therfore, the pair $\left(\varinjlim\limits_{\alpha \in \Lambda} \mathcal{X}_{\alpha}, \{\psi_\alpha\}_{\alpha \in \Lambda} \right)$ is \textit{compatible} with the inductive system $\left(\{{\mathcal{X}}_{\alpha}\}_{\alpha \in \Lambda},\{\psi_{\beta, \alpha}\}_{\alpha\leq\beta}\right)$ in the sense that Equation \eqref{Equation: inductivecompatibility} holds true. Equivalently, we have the following commuting diagram:
\begin{equation*}
\begin{tikzcd}[every matrix/.append style={name=m},
  execute at end picture={
        \draw [<-] ([xshift=0mm,yshift=2mm]m-2-2.north) arc[start angle=-90,delta angle=270,radius=0.25cm];
  }]
    \mathcal{X}_{\alpha} \arrow{dr}[swap]{\psi_{\beta,\alpha}} \arrow{rr} {\psi_{\alpha}} && \mathcal{X} \\
    & \mathcal{X}_{\beta} \arrow{ur}[swap]{\psi_{\beta}}
\end{tikzcd} \newline
\end{equation*}
 Let us consider $(\mathcal{Y},\{\kappa_\alpha\}_{\alpha \in \Lambda}),$ where $\mathcal{Y}$ is a locally convex space and $\kappa_\alpha \colon \mathcal{X}_\alpha \rightarrow \mathcal{Y}$ is a continuous linear map, for each $\alpha \in \Lambda$. Suppose that $(\mathcal{Y},\{\kappa_\alpha\}_{\alpha \in \Lambda})$ is compatible with the inductive system $\left(\{{\mathcal{X}}_{\alpha}\}_{\alpha \in \Lambda},\{\psi_{\beta, \alpha}\}_{\alpha\leq\beta}\right),$ which means   
    $\kappa_\alpha = \kappa_\beta \, \circ \, \psi_{\beta,\alpha}, \text{ whenever }\alpha \leq \beta,$ then we have the following observations.
    \begin{enumerate}
\item we can define a continuous linear map $\kappa \colon \varinjlim\limits_{\alpha \in \Lambda} \mathcal{X}_{\alpha} \to \mathcal{Y}$ by 
\begin{equation*}
    \kappa \big(\{x_{\alpha}\}_{\alpha} + \mathcal{X}_{0} \big) = \sum\limits_{\ell=1}^{n} \kappa_{\alpha_{\ell}}(x_{\alpha_{\ell}}),
\end{equation*}
where $\{\alpha_{\ell} \in \Lambda:\; 1 \leq \ell \leq n\}$ is the support of $\{x_{\alpha}\}_{\alpha \in \Lambda}$. Further, for each $\alpha \in \Lambda$ with $x_{\alpha} \in \mathcal{X}_{\alpha}$ the map satisfy that 
\begin{equation*}
    \big(\kappa \circ \psi_{\alpha}\big)(x_{\alpha}) = \kappa \big( \delta_{x_{\alpha}}+ \mathcal{X}_{0}\big) = \kappa_{\alpha}(x_{\alpha}).
\end{equation*}
Equivalently, we have the following commuting diagram.
\begin{equation*}
    \begin{tikzcd}[every matrix/.append style={name=m},
  execute at end picture={
        \draw [<-] ([xshift=0mm,yshift=2mm]m-2-2.north) arc[start angle=-90,delta angle=270,radius=0.25cm];
  }]
    \mathcal{Y} \arrow{rr} {\kappa} && \mathcal{X} \\
    & \mathcal{X}_{\alpha} \arrow{ul}{\kappa_{\alpha}} \arrow{ur}[swap]{\psi_{\alpha}}
\end{tikzcd} \newline
\end{equation*}
        \item The map $\kappa$ defined above is unique. Assume that there is another continuous linear map $\theta \colon \varinjlim\limits_{\alpha\in \Lambda}\mathcal{X}_\alpha \rightarrow \mathcal{Y} $ satisfying,
        \begin{equation}
            \theta \circ \psi_\alpha = \kappa_\alpha, \text{ for all } \alpha \in \Lambda.
        \end{equation} 
        If $\{\alpha_{\ell} \in \Lambda \colon 1 \leq \ell \leq n\}$ is the support of the vector $\{x_\alpha\}_{\alpha \in \Lambda}+\mathcal{X}_0 \in \varinjlim\limits_{\alpha\in \Lambda}\mathcal{X}_\alpha,$ then
\begin{align*}
 \theta\big(\{x_\alpha\}_{\alpha \in \Lambda}+\mathcal{X}_0\big) = \theta\left(\sum\limits^{n}_{\ell=1}\delta_{x_{\alpha_\ell}}+\mathcal{X}_0\right) &=\sum\limits^{n}_{\ell=1}\theta\big(\delta_{x_{\alpha_\ell}}+\mathcal{X}_0\big) \\
            &=\sum\limits^{n}_{\ell=1}\theta\big(\psi_{\alpha_\ell}(x_{\alpha_\ell})\big) \\
            &=\sum\limits^{n}_{\ell=1}\kappa_{\alpha_\ell}(x_{\alpha_\ell}) \\
            &=\kappa\big(\{x_\alpha\}_{\alpha \in \Lambda}+\mathcal{X}_0\big).
        \end{align*}
        Since $\{x_{\alpha}\}_{\alpha \in \Lambda}$ is arbitrary, we get that $\theta=\kappa.$
    \end{enumerate}
    
    \noindent Now, we conclude from observations $(1),(2)$ that, the inductive limit $\left(\varinjlim\limits_{\alpha \in \Lambda} \mathcal{X}_{\alpha}, \{\psi_\alpha\}_{\alpha \in \Lambda} \right)$ is unique upto compatibility. This means, if there is another pair $(\mathcal{Y},\{\kappa_\alpha\}_{\alpha \in \Lambda})$ compatible with the inductive system $\left(\{{\mathcal{X}}_{\alpha}\}_{\alpha \in \Lambda},\{\psi_{\beta, \alpha}\}_{\alpha\leq\beta}\right),$ then by Observation $(2),$ there is a unique continuous linear map $\kappa \colon \varinjlim\limits_{\alpha \in \Lambda} \mathcal{X}_{\alpha} \to \mathcal{Y}$ such that $\kappa \circ \psi_\alpha=\kappa_\alpha,$ for all $\alpha \in \Lambda.$

 The following definition describes certain class of maps known as {\it coherent map} between inductive limits of two inductive systems.  
 \begin{defn}\label{defn: coherent inj}
     Let $(\Lambda,\leq)$ be a directed POSET. Suppose ($\mathcal{X}, \{\psi_\alpha\}_{\alpha \in \Lambda}$) and $(\mathcal{Y}, \{\kappa_\alpha\}_{\alpha \in \Lambda})$ are inductive limits of inductive systems $(\{\mathcal{X}_\alpha\}_{\alpha \in \Lambda}, \{\psi_{\beta,\alpha}\}_{\alpha \leq \beta})$ and $(\mathcal{Y}_\alpha, \{\kappa_{\beta,\alpha}\}_{\alpha \leq \beta} )$ of locally convex spaces respectively. A linear map $g \colon \mathcal{X} \rightarrow \mathcal{Y}$ is said to be \textit{coherent} if there exists a net $\{g_\alpha\}_{\alpha \in \Lambda}$ of linear maps from $\mathcal{X}_\alpha$ to $\mathcal{Y}_\alpha$, such that $g\circ \psi_\alpha = \kappa_\alpha \circ g_\alpha$ for all $\alpha \in \Lambda.$ Equivalently,
 
\begin{equation*}
\begin{tikzcd}[every matrix/.append style={name=m},
    execute at end picture={
    \draw [<-] ([xshift=4.5mm,yshift=3mm]m-2-2.west) arc[start angle=-90,delta angle=270,radius=0.33cm]; }] 
\mathcal{X}  \arrow{rr}{g}&& \mathcal{Y} \\
&&\\
\mathcal{X}_\alpha \arrow{uu}{\psi_\alpha} \arrow[rr,swap,"g_\alpha"]  && \mathcal{Y}_\alpha \arrow[uu,swap,"\kappa_\alpha"]
\end{tikzcd}
\end{equation*}
\end{defn}
 Now we illustrate the notion of inductive limit with an example below. 
 \begin{eg}\label{example: ind lim eg} Let $\Lambda = \mathbb{N} $ and $C\big(\big[\frac{-1}{n},\frac{1}{n}\big]\big)$ denote the class of complex valued  continuous functions on $[\frac{-1}{n},\frac{1}{n}]$, for $n \in \mathbb{N}$.  For each $m \leq n$ the map $\psi_{n,m} \colon C\big(\big[\frac{-1}{m},\frac{1}{m}\big]\big) \rightarrow C\big([\frac{-1}{n},\frac{1}{n}]\big)$ given by
 \begin{equation*}
     \psi_{n,m}(f) = f\big|_{[\frac{-1}{n},\frac{1}{n}]}, \text{ for all } f \in C\big(\big[\frac{-1}{m},\frac{1}{m}\big]\big)
 \end{equation*}
 is a continuous linear map. Clearly, $\Big( \big\{ C\big(\big[\frac{-1}{n},\frac{1}{n}\big]\big)\big\}_{n \in \mathbb{N}},\; \{\psi_{n,m}\}_{m \leq n} \Big)$ is an inductive system of locally convex spaces. Now we compute its inductive limit. 
If we define $\psi_n \colon C\big([\frac{-1}{n},\frac{1}{n}]\big) \rightarrow \mathbb{C}$ by
 \begin{equation*}
     \psi_n(f) = f(0), \text{ for all } f \in C\big(\big[\frac{-1}{n},\frac{1}{n}\big]\big)
 \end{equation*}
 then $\psi_n$ is a continuous linear map, for all $n \in \mathbb{N}.$ Further, for $m \leq n$ we have $\psi_n \circ \psi_{n,m} = \psi_m.$  This shows that  $\big(\mathbb{C},\{\psi_{n}\}_{n \in \mathbb{N}}\big)$ is compatible with the inductive system $\Big( \big\{ C\big(\big[\frac{-1}{n},\frac{1}{n}\big]\big)\big\}_{n \in \mathbb{N}},\; \{\psi_{n,m}\}_{m \leq n} \Big)$ and hence $\big(\mathbb{C},\{\psi_{n}\}_{n \in \mathbb{N}}\big)$ is the inductive limit of the inductive system $\Big( \big\{ C\big(\big[\frac{-1}{n},\frac{1}{n}\big]\big)\big\}_{n \in \mathbb{N}},\; \{\psi_{n,m}\}_{m \leq n} \Big)$. Therefore, by the uniqueness of inductive limit (upto compatibility) we have 
 \begin{equation*}
     \varinjlim\limits_{n \in \mathbb{N}}C\big(\big[\frac{-1}{n},\frac{1}{n}\big]\big)=\mathbb{C}.
 \end{equation*}
  \end{eg}
\begin{note}
    In the Definition \ref{defn: ind system}, if we consider $\mathcal{X}_{\alpha} \subseteq \mathcal{X}_{\beta}$ and $\psi_{\beta, \alpha} \colon \mathcal{X}_{\alpha} \to \mathcal{X}_{\beta}$ is an inclusion map (i.e.,\; $\psi_{\beta, \alpha}(x) = x$ for all $x \in \mathcal{X}_{\alpha}$) whenever $\alpha \leq \beta$ then the inductive system $\left(\{{\mathcal{X}}_{\alpha}\}_{\alpha \in \Lambda},\{\psi_{\beta, \alpha}\}_{\alpha\leq\beta}\right)$ is called a {\it strict inductive system}. For a strict inductive system $\left(\{{\mathcal{X}}_{\alpha}\}_{\alpha \in \Lambda},\{\psi_{\beta, \alpha}\}_{\alpha\leq\beta}\right)$, we see that $\mathcal{X}_{0} = \{0\}$ from Equation \eqref{equation: X0} and by following Equation \eqref{Equation: indlimit}, we have 
  \begin{equation*}
     \varinjlim_{\alpha \in \Lambda}\mathcal{X}_\alpha = (\bigoplus_{\alpha \in \Lambda}\mathcal{X}_\alpha)/\mathcal{X}_0 = \bigcup_{\alpha \in \Lambda}\mathcal{X}_\alpha.
 \end{equation*}
\end{note} 
Now we recall the definition of locally Hilbert space \cite{gheondea}.  
\begin{defn}\label{defn: loc hilbert sp}
 A family $\{\mathcal{H}_\alpha\}_{\alpha \in \Lambda}$ of Hilbert spaces is said to be a strictly inductive system  if
  \begin{enumerate}
      \item[(a)] $(\Lambda, \leq)$ is a directed POSET,
      \item[(b)] $\{\mathcal{H}_\alpha\}_{\alpha \in \Lambda}$ is a net of Hilbert spaces where the inner product on $\mathcal{H}_\alpha$ is denoted by $ \langle\cdot,\cdot\rangle_{_\alpha},$ for $ \alpha \in \Lambda,$
      \item[(c)]  $\mathcal{H}_\alpha \subseteq \mathcal{H}_\beta,$ whenever $\alpha \leq \beta,$ 
      \item[(d)] Whenever $\alpha \leq \beta,$ the inclusion map $I_{\beta,\alpha} : \mathcal{H}_\alpha \rightarrow \mathcal{H}_\beta$ is isometric, i.e.
      \begin{equation*}
          \langle x,y \rangle_{\alpha} = \langle x,y \rangle_{\beta}, \text{ for all } x,y \in \mathcal{H}_\alpha.
      \end{equation*}
  \end{enumerate}
  \end{defn}
  For the strict inductive system $\{\mathcal{H}_\alpha\}_{\alpha  \in \Lambda}$ of Hilbert spaces, the inductive limit 
  \begin{equation*}
      \mathcal{D} = \displaystyle \varinjlim_{\alpha \in \Lambda}\mathcal{H}_\alpha = \bigcup_{\alpha \in \Lambda}\mathcal{H}_\alpha
      \end{equation*}
      is called a \textit{locally Hilbert space.} 
      
      In this work, we use the following terminology for locally Hilbert space, namely \textit{quantized domain} \cite[Definition 2.3]{SanthoshKP}. A quantized domain in a Hilbert space $\mathcal{H}$ is represented with a triple $\{\mathcal{H}; \mathcal{E}; \mathcal{D}\},$ where  $\mathcal{E}=\{\mathcal{H}_\alpha\}_{\alpha \in \Lambda}$ be an upward filtered family (or strictly inductive system) of closed subspaces of $\mathcal{H}$, $\mathcal{D}=\bigcup\limits_{\alpha \in \Lambda}\mathcal{H}_\alpha$ and $\mathcal{H}= \overline{\mathcal{D}}.$ The following example is motivated from \cite[Example 2.9]{SanthoshKP} and gives clear description of Frechet domain (when $\Lambda=\mathbb{N}$) of Hilbert spaces \cite{dosiev}.
   
  
  \begin{eg}\label{example: loc hilbert eg}
  Let $\mathcal{H} = \ell^{2}(\mathbb{N})$ and $\{e_{n}: n \in \mathbb{N}\}$ be an orthonormal basis of $\mathcal{H}$. Consider $\mathcal{H}_{n}:= span\{e_{1}, e_{2}, \cdots, e_{n}\},\; \forall \; n \in \mathbb{N}.$ Then each $\mathcal{H}_n$ is a finite dimensional Hilbert space (closed subspaces of $\ell^{2}(\mathbb{N})$) for all $n \in \mathbb{N}$. Further, $\mathcal{E} := \{\mathcal{H}_n \colon n\in \mathbb{N}\}$ is a strict inductive system or an upward filtered family of Hilbert spaces since $\mathcal{H}_n \subseteq \mathcal{H}_{n+1}, \text{ for all } n \in \mathbb{N}.$ Then the inductive limit is given by 
  \begin{equation*}
      \mathcal{D} = \displaystyle \bigcup_{n \in \mathbb{N}}\mathcal{H}_n = \text{span}\{e_i \colon i \in \mathbb{N}\},
  \end{equation*} 
  which is dense in $\ell^{2}(\mathbb{N}).$ Therefore, the triple $\big\{\ell^{2}(\mathbb{N});\mathcal{E};\mathcal{D}\big\}$ is a quantized domain in the Hilbert space  $\ell^{2}(\mathbb{N})$.
\end{eg}
Note that, every Hilbert space is a locally Hilbert space, but the converse need not be true (see Example \ref{example: loc hilbert eg}).\newline
We recall the notion of \textit{locally von Neumann algebra} using the notations of quantized domain. For a detailed discussion, the reader is directed to \cite{Joita}. 
\begin{defn}\label{defn: locally von neumann alg}
    Let $\{\mathcal{H}; \mathcal{E}; \mathcal{D}\}$ be a quantized domain associated to the Hilbert space $\mathcal{H}.$
    \begin{enumerate}
        \item If $\mathcal{L}(\mathcal{D})$ denotes the set of all linear operators on $\mathcal{D}$, \,then the $\ast$-algebra of all \textit{non-commutative continuous functions} on $\{\mathcal{H}; \mathcal{E}; \mathcal{D}\}$ is defined in \cite{dosiev} as,
\begin{equation*}
C^{\ast}_{\mathcal{E}}(\mathcal{D}): = \big\{ T \in \mathcal{L}(\mathcal{D}):\; P_{\mathcal{H}_\alpha}T \subseteq TP_{\mathcal{H}_\alpha}\; \&\; TP_{\mathcal{H}_\alpha} \in \mathcal{B}(\mathcal{H}_{\alpha}), \text{ for all } \alpha \in \Lambda \big\},
\end{equation*}
where $P_{\mathcal{H}_\alpha}$ is the projection of $\mathcal{H}$ onto $\mathcal{H}_\alpha$.

It is worth to mention that for every locally $C^{\ast}$-algebra $\mathcal{A},$ there exists a quantized domain $\{\mathcal{H}; \mathcal{E}; \mathcal{D}\}$ for which there is a local isometrical $\ast$-homomorphism from $\mathcal{A}$ to $C^{\ast}_{\mathcal{E}}(\mathcal{D})$ (see Theorem 7.2 of \cite{dosiev}).

        \item If $x\in \mathcal{D},$ then there exists $\alpha \in \Lambda$ such that $x \in \mathcal{H}_\alpha$ and $q_{x}(T)=\|T(x)\|_{\mathcal{H}_\alpha},$ for all $T \in C^{\ast}_{\mathcal{E}}(\mathcal{D})$ defines a seminorm.
        The \textit{strong operator topology} on $C^{\ast}_{\mathcal{E}}(\mathcal{D})$ is the locally convex topology induced by the family $\{q_x \colon x \in \mathcal{D}\}$ of seminorms.  
        \item A locally $C^{\ast}$-subalgebra of $C^{\ast}_{\mathcal{E}}(\mathcal{D})$ is called \textit{locally von Neumann algebra} if it is closed under the strong operator topology. An equivalent description is given in \cite[Proposiotion 3.14]{Joita} that, every locally von Neumann algebra is the projective limit of a projective system of von Neumann algebras.
    \end{enumerate}
     
\end{defn}
\begin{rmk}\label{rmk: locally bounded operator}
    If $\mathcal{E}=\{\mathcal{H}_\alpha \colon \alpha \in \Lambda\}$ and $\mathcal{D}=\displaystyle\bigcup_{\alpha \in \Lambda}\mathcal{H}_\alpha,$ then $\big(\{\mathcal{B}(\mathcal{H}_\alpha)\}_{\alpha \in \Lambda},\{\phi_{\alpha,\beta}\}_{\alpha \leq \beta}\big)$ is a projective system of $C^{\ast}$-algebras, where $\phi_{\alpha,\beta} \colon \mathcal{B}(\mathcal{H}_\beta) \rightarrow \mathcal{B}(\mathcal{H}_\alpha)$ is a restriction map, whenever $\alpha \leq \beta.$ Clearly, $\phi_{\alpha,\beta}$ is a $C^{\ast}$-representation. For each $\alpha,\; \phi_{\alpha}(T)=T\big|_{\mathcal{H}_\alpha}$ defines a $\ast$-homomorphism from $C^{\ast}_{\mathcal{E}}(\mathcal{D})$ to $\mathcal{B}(\mathcal{H}_\alpha)$ such that $\phi_{\alpha,\beta} \circ \phi_{\beta}=\phi_\alpha$ whenever $\alpha \leq \beta.$ In other words, $\left(C^{\ast}_{\mathcal{E}}(\mathcal{D}),\{\phi_\alpha\}_{\alpha \in \Lambda}\right)$ is compatible with $\left(\{\mathcal{B}(\mathcal{H}_\alpha)\}_{\alpha \in \Lambda},\{\phi_{\alpha,\beta}\}_{\alpha \leq \beta}\right).$ It follows that, $C^{\ast}_{\mathcal{E}}(\mathcal{D})=\varprojlim\limits_{\alpha \in \Lambda}\mathcal{B}(\mathcal{H}_\alpha)$ is a locally $C^{\ast}$-algebra. In view of this, \text{for every}\; $T \in  C^{\ast}_{\mathcal{E}}(\mathcal{D})$, we denote it by 
    \begin{equation*}
    T=\varprojlim\limits_{\alpha \in \Lambda}T_{\alpha},\; \text{where}\; T_{\alpha} = T\big|_{\mathcal{H}_{\alpha}},\; \alpha \in \Lambda.
    \end{equation*}
\end{rmk}
 \section{Character space of a locally $C^{\ast}$-algebra}\label{section: loc maximal ideal sp} \label{Section 2}
 In this section, our aim is to define the character space of a commutative unital locally $C^{\ast}$-algebra and study its topological properties. In the literature, the authors of \cite{michael,schmudgen} have made a few remarks about the spectrum of an element of a locally $C^{\ast}$-algebra. However, there is no explicit description of character space in this setting. Here we define the class and describe it in full details. 
 
 Let $\mathcal{A}$ be a commutative unital locally $C^{\ast}$-algebra and $\{p_{\alpha}\}_{\alpha \in \Lambda}$ be a family of $C^{\ast}$-seminorms. Then by taking, $\mathcal{I}_{\alpha}:= \{a\in \mathcal{A}:\; p_{\alpha}(a)=0\}$ and $\mathcal{A}_{\alpha}:= \mathcal{A}/\mathcal{I}_\alpha, $ we have seen that $\mathcal{A}=\varprojlim\limits_{\alpha\in \Lambda}\mathcal{A}_\alpha,$ the projective limit of commutative unital $C^{\ast}$-algebras $\{\mathcal{A}_{\alpha}\}_{\alpha \in \Lambda}$. Recall that the projective limit topology on $\mathcal{A}$ is the smallest locally convex topology on $\mathcal{A}$ such that the quotient map 
 $\pi_\alpha \colon \mathcal{A} \rightarrow \mathcal{A}_\alpha$ is continuous, for all $\alpha \in \Lambda.$ From now onwards, $\mathcal{A}$ denotes a commutative unital locally $C^{\ast}$-algebra unless it is specified otherwise.
 
 For each $\alpha \in \Lambda,$ the maximal ideal space of $\mathcal{A}_{\alpha}$ is denoted by
 \begin{equation*}
     \mathcal{M}_{\mathcal{A}_\alpha} = \big\{\varphi_\alpha : \mathcal{A}_\alpha \rightarrow \mathbb{C} \colon \varphi_\alpha \text{ is multiplicative and linear}\big\}.
 \end{equation*}
Recall that, $\mathcal{M}_{\mathcal{A}_\alpha} $ is a non-empty weak$^{\ast}$-compact in $\mathcal{A}^{\ast}_\alpha$ and it has one to one correspondence with all maximal ideals of $\mathcal{A}_\alpha$ (for details, see \cite{kehezhu}). Before we propose the notion of character space, let us understand the behaviour of multiplicative linear functional defined on $\mathcal{A}.$ Firstly, note that for any $\alpha \in \Lambda$ and $\varphi_{\alpha} \in \mathcal{M}_{\mathcal{A}_{\alpha}}$, the map $\varphi_{\alpha} \circ \pi_{\alpha}$ is a multiplicative linear functional on $\mathcal{A}$ and it is continuous since both $\varphi_{\alpha}$ and $\pi_{\alpha}$ are continuous. This shows that the collection of multiplication linear functionals on $\mathcal{A}$ is non-empty. Precisely, each multiplicative linear functional on a commutative unital $C^{\ast}$-algebra $\mathcal{A}_{\alpha}$ induces a continuous multiplicative linear functional on the commutative unital locally $C^{\ast}$-algebra $\mathcal{A}.$ 

Note that, a multiplicative linear functional on $\mathcal{A}$  is not necessarily continuous. We give an example below. 
\begin{eg}\label{eg: Phi is not continuous}
    Let $\mathcal{A} = C\big([0, \infty)\big)$ denote the class of complex valued continuous functions on $[0, \infty)$ and let $\Lambda = \mathbb{N}.$ Consider that $\mathcal{A}$ is equipped with the locally convex topology induced by the family $\{p_{n}\}_{n \in \mathbb{N}}$ of seminorms, where 
    \begin{equation*}
        p_{n}(f) = \sup\left\{ |f(t)|:\; \frac{1}{2} \leq t \leq n\right\}.
    \end{equation*}
    Here $\mathcal{A}$ is a locally convex commutative unital $\ast$-algebra. Since for every $f \in \mathcal{A},$
    \begin{equation*}
        p_{n}(f^{\ast}f) = \sup\left\{ |f(t)|^{2}:\; \frac{1}{2} \leq t \leq n\right\} = \sup\left\{ |f(t)|:\; \frac{1}{2} \leq t \leq n\right\}^{2} = P_{n}(f)^{2}
    \end{equation*} 
    and whenever $m \leq n$, we have 
    \begin{equation*}
        p_{m}(f) = \sup\left\{ |f(t)|:\; \frac{1}{2} \leq t \leq m\right\} \leq  \sup\left\{ |f(t)|:\; \frac{1}{2} \leq t \leq n\right\} = p_{n}(f),
    \end{equation*}
    it follows that $\{p_{n}\}_{n \in \mathbb{N}}$ is an upward filtered family of $C^{\ast}$-seminorms on $\mathcal{A}.$ Firstly, note that for $m \leq n,$ the map $\phi_{m,n}\colon C\left([\frac{1}{2}, n]\right) \to C\left([\frac{1}{2}, m]\right)$ given by $f \mapsto f\big|_{[\frac{1}{2}, m]}$ is a surjective $\ast$-homomorphism of $C^{\ast}$-algebras. It is easy to see that $\left(\big\{C\left([\frac{1}{2},n]\right)\big\}_{n \in \mathbb{N}},\{\phi_{m,n}\}_{m \leq n}\right)$ forms a projective system of commutative unital $C^{\ast}$-algebras. Similarly, for each $n \in \mathbb{N}$ we define $\phi_{n} \colon \mathcal{A} \to C\left([\frac{1}{2},n]\right)$ by 
    \begin{equation*}
        \phi_{n}(f) = f\big|_{[\frac{1}{2}, n]},\; \text{for all}\; f \in \mathcal{A}
    \end{equation*}
    is a continuous $\ast$-homomorphism satisfying, $\phi_{m,n}\circ \phi_{n} = \phi_{m}$ for $m \leq n.$ That is, $\left( \mathcal{A}, \{\phi_{n}\}_{n\in \mathbb{N}} \right)$ is compatible with the projective system $\left(\big\{C[\frac{1}{2},n]\big\}_{n \in \mathbb{N}},\{\phi_{m,n}\}_{m \leq n}\right).$ Therefore, $\mathcal{A}$ is a commutative unital locally $C^{\ast}$-algebra. 

    Now consider the map $\Phi\colon \mathcal{A} \to \mathbb{C}$ given by $\Phi(f) = f(0)$, for all $f \in \mathcal{A}$. Clearly, $\Phi$ is a multiplicative linear functional. We claim that $\Phi$ is not continuous. For this, let us consider the open ball $B\left(0, \frac{1}{2}\right)$ of radius $\frac{1}{2}$ around $0$ in $\mathbb{C}$ and see that 
    \begin{equation*}
        \Phi^{-1}\left( B \left(0, \frac{1}{2}\right) \right) = \left\{ x \in \mathcal{A}:\; |\Phi(x)|< \frac{1}{2}\right\}.
    \end{equation*}
    Now we show that $ \Phi^{-1}\left( B \left(0, \frac{1}{2}\right) \right)$ is not open in Projective limit topology. For every $\ell \in \mathbb{N}$, define $g_{\ell} \colon [0, \infty) \to \mathbb{C} $ by 
    \begin{equation*}
        g_{\ell}(t) = \frac{1}{2 + t\ell^{2}}, \; \text{for}\; 0 \leq t < \infty.
    \end{equation*}
    Note that each $g_{\ell}$ is a continuous decreasing function. For every $n \in \mathbb{N},$ we have 
    \begin{equation*}
        p_{n}(g_{\ell}) = \sup\left\{ |g_{\ell}(t)|:\; \frac{1}{2} \leq t \leq n\right\} = g_{\ell}\left(\frac{1}{2}\right) = \frac{2}{4+\ell^{2}} < \frac{1}{\ell}.
    \end{equation*}
    Further, $g_{\ell}(0) = \frac{1}{2}$ and so, $g_{\ell} \notin  \Phi^{-1}\left( B \left(0, \frac{1}{2}\right) \right).$ In summary, we have shown that for every $\ell, n \in \mathbb{N}$ there is a $g_{\ell} \in p_{n}^{-1}\left( B(0, \frac{1}{\ell})\right)$ such that $g_{\ell} \notin  \Phi^{-1}\left( B \left(0, \frac{1}{2}\right) \right).$ Equivalently, $p_{n}^{-1}\left( B(0, \frac{1}{\ell})\right)$ is not entirely contained in $\Phi^{-1}\left( B \left(0, \frac{1}{2}\right) \right)$ for every $\ell, n \in \mathbb{N}$. Hence $\Phi$ is not continuous. 
    \end{eg}
In the following graph, we draw the functions $g_{1}, g_{2}, g_{3}$ that are defined in Example \ref{eg: Phi is not continuous} to display their nature. 
    
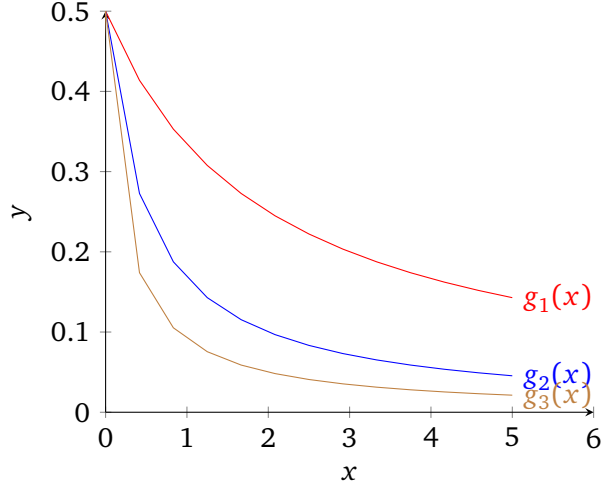
\begin{equation*}
\begin{tikzpicture}
    \begin{axis}[xmin=0, xmax=6, ymin=0, ymax= 1/2, axis lines=left, xlabel=$x$, ylabel=$y$,]
        \addplot[color=red]{1/(2+x)}
        node[right,pos=1]{$g_1(x)$};
        \addplot[color=blue]{1/(2+4*x)}
        node[right,pos=1]{$g_2(x)$};
        \addplot[color=brown]{1/(2+9*x)}
        node[right,pos=1]{$g_{3}(x)$};
    \end{axis}
\end{tikzpicture}
\end{equation*}
Now we show that certain continuous multiplicative linear functional on a commutative unital locally $C^{\ast}$-algebra must be induced from multiplicative linear functional on some quotient $C^{\ast}$-algebra. 
\begin{thm}\label{thm: Phi associated to varphialpha}
    Let $\mathcal{A}$ be a commutative unital locally $C^{\ast}$-algebra and $\Phi \colon \mathcal{A} \rightarrow \mathbb{C}$ be a multiplicative linear functional. If $\Phi$ is continuous and for some $\alpha_{0} \in \Lambda$, 
    \begin{equation} \label{Eq: sup Phi = 1}
        \sup\big\{ |\Phi(x)|:\; x\in \mathcal{A}, \; p_{\alpha_{0}}(x)<1\big\} = 1.
    \end{equation}
    then the map given by $\varphi_{\alpha_{0}}\left( \pi_{\alpha_{0}}(x)\right) = \Phi(x),\; x \in \mathcal{A},$ is a well-defined multiplicative linear functional on $\mathcal{A}_{\alpha_{0}}$. In other words,  $\Phi = \varphi_{\alpha_{0}} \circ \pi_{\alpha_{0}}.$
\end{thm}
\begin{proof} From the hypothesis, it is clear that $|\Phi(x)|\leq 1$ whenever $p_{\alpha_{0}}(x)<1$. Assume that $p_{\alpha_{0}}(x) = 0$, then for any $n \in \mathbb{N}$ we have $p_{\alpha_{0}}(nx) = 0 <1$. It follows that 
\begin{equation*}
    |\Phi(x)| = \frac{1}{n} |\Phi(nx)| \leq \frac{1}{n}, \; \text{for every}\; n \in \mathbb{N}.  
\end{equation*}
So, $\Phi(x) = 0.$ The map $\varphi_{\alpha_{0}}:\; \mathcal{A}_{\alpha_{0}} \to \mathbb{C}$ given by 
\begin{equation*}
    \varphi_{\alpha_{0}}(\pi_{\alpha_{0}}(x)) = \Phi(x),\; \text{for all}\; x \in \mathcal{A}
\end{equation*}
is well-defined linear map. Further, 
\begin{equation*}
    \varphi_{\alpha_{0}}\left(\pi_{\alpha_{0}}(x) \pi_{\alpha_{0}}(y) \right) = \varphi_{\alpha_{0}} \left( \pi_{\alpha_{0}}(xy)\right) = \Phi(xy) = \phi(x) \Phi(y) =\varphi_{\alpha_{0}}\left(\pi_{\alpha}(x) \right) \;\varphi_{\alpha_{0}}\left(\pi_{\alpha}(y) \right),
\end{equation*}
for any $x,y \in \mathcal{A}$ and 
\begin{equation*}
    \sup\big\{\left|\varphi_{\alpha_{0}}\left(\pi_{\alpha_{0}}(x) \right) \right|:\; x \in \mathcal{A},\; \|\pi_{\alpha_{0}}(x)\| <1 \big\} = \sup\big\{ |\Phi(x)|:\; x\in \mathcal{A}, \; p_{\alpha_{0}}(x)<1\big\} = 1.
\end{equation*}
Therefore, $\varphi_{\alpha_{0}} \in \mathcal{M}_{\mathcal{A}_{\alpha_{0}}}.$
\end{proof}
\begin{rmk}
    If $(\Lambda, \leq)$ is a totally ordered set, the the condition given in Equation \ref{Eq: sup Phi = 1} of Theorem \ref{thm: Phi associated to varphialpha} is redundant. That is, for a continuous multiplicative linear functional $\Phi$ defined on a commutative unital locally $C^{\ast}$-algebra $\mathcal{A}$, there exists an $\alpha_{0} \in \Lambda$ such that Equation \ref{Eq: sup Phi = 1} superfluous. This can be seen as follows: since $\Phi$ is continuous, there exists an $\epsilon > 0$ and $\{\alpha_{1}, \alpha_{2}, \cdots, \alpha_{N}\} \subset \Lambda$ satisfying
    \begin{equation*}
        \bigcap\limits_{i=1}^{N} p_{\alpha_{i}}^{-1}\left( B(0, \epsilon)\right) \subseteq \Phi^{-1}\left(B(0,1) \right).
    \end{equation*}
    Since $\Lambda$ is a totally ordered set, by choosing $\alpha_{0} = \min\{\alpha_{{1}}, \alpha_{{2}}, \cdots, \alpha_{{N}}\}$ we have 
    \begin{equation}\label{Eq: epsilon 1}
        |\Phi(x)|<1\; \text{ whenever} \; p_{\alpha_{0}}(x) < \epsilon.
    \end{equation}
    Equivalently, $\Phi\left( p_{\alpha_{0}}^{-1}\left(B(0, \epsilon)\right)\right)$ is bounded in $\mathbb{C}.$ Now we claim that Equation \ref{Eq: sup Phi = 1} holds true. If $\epsilon \geq 1$ then the result follows from Equation \eqref{Eq: epsilon 1}. Suppose that  $\epsilon <1$ and there exists an $x \in \mathcal{A}$ such that $p_{\alpha_{0}}(x)<1$ but $|\Phi(x)|>1$ then 
    \begin{equation*}
        p_{\alpha_{0}}\left( \epsilon x^{n}\right) = \epsilon p_{\alpha_{0}}\left( x^{n}\right) \leq \epsilon p_{\alpha_{0}}(x)^{n} < \epsilon,
    \end{equation*}
     but the sequence $\{\Phi(\epsilon x^{n})\}_{n \in \mathbb{N}}$ is not bounded since $|\Phi(\epsilon x^{n})| = \epsilon |\Phi(x)|^{n}.$ Therefore, Equation \eqref{Eq: epsilon 1} holds true for $\epsilon>0$. Hence $\varphi_{\alpha_{0}}$ defined as in Theorem \ref{thm: Phi associated to varphialpha} is in $\mathcal{M}_{\mathcal{A}_{\alpha_{0}}}$ and $\Phi = \varphi_{\alpha_{0}}\circ\pi_{\alpha_{0}}.$    
     \end{rmk}
It is worth to point out that even when a multiplicative linear functional $\Phi \colon \mathcal{A} \to \mathbb{C}$ is continuous, the map $\varphi_{\alpha}$ given in Theorem \ref{thm: Phi associated to varphialpha} may not be well defined for every $\alpha \in \Lambda.$ We illustrate this situation with the an example below. 
\begin{eg}\label{eg: Phi not associated to varphi1}
    Let us consider the commutative unital locally $C^{\ast}$- algebra $C(\mathbb{R}).$ As described in Example \ref{Example: loc C* alg}, we see that
    \begin{equation*}
        C(\mathbb{R})= \varprojlim\limits_{n \in \mathbb{N}} C\left([-n, n]\right)
    \end{equation*}
    In fact, $\mathcal{A}$ is equipped with the locally convex topology induced by the upward filtered family $\{p_n\}_{n \in \mathbb{N}}$ of $C^{\ast}$-seminorms, where
    \begin{equation*}
        p_n(f)=\sup\big\{|f(t)| \colon -n \leq t \leq n\big\}.
    \end{equation*}
    Following Remark \ref{rmk: Invlimit}, there is a canonical quotient map $\pi_n \colon C(\mathbb{R}) \rightarrow C(\mathbb{R})/\mathcal{I}_n$ is continuous for each $n \in \mathbb{N}$. Here $\mathcal{I}_n=\big\{f \in C(\mathbb{R}) \colon p_n(f)=0\big\}$ is a closed two sided $\ast$-ideal of $\mathcal{A}$ for all $n \in \mathbb{N}.$ Now define $\Phi \colon C(\mathbb{R}) \rightarrow \mathbb{C}$ by
    \begin{equation*}
        \Phi(f)=f(2), \text{ for all } f \in C(\mathbb{R}).
    \end{equation*}
    Then $\Phi$ is clearly a multiplicative linear functional on $C(\mathbb{R}).$ Further, if we define $\varphi_{2} \colon C(\mathbb{R})/\mathcal{I}_{2} \to \mathbb{C}$ by 
    \begin{equation*}
        \varphi_{2}\left( \pi_{2}(f)\right) =   f(2), \; \text{for all}\; f \in C(\mathbb{R})
    \end{equation*}
    then $\varphi_{2}$ a multiplicative linear functional on the commutative unital $C^{\ast}$-algebra $C(\mathbb{R})/\mathcal{I}_{2}$ and so, $\Phi = \varphi_{2} \circ \pi_{2}$ is continuous. Indeed, the map $\varphi_{n}\colon C(\mathbb{R})/\mathcal{I}_{n} \to \mathbb{C}$ given by $ \varphi_{n}\big( \pi_{n}(f)\big) = \Phi(f)$ is well-defined and $\Phi = \varphi_{n}\circ \pi_{n}$ for $n\geq 2.$ However, we show that $\varphi_{1}\colon C(\mathbb{R})/\mathcal{I}_{1} \to \mathbb{C}$ given by $\varphi_1\left(\pi_1(f)\right)=f(2)$ is not well defined. For if consider the following two continuous functions, 
    \begin{equation*}
        f(t)=t,\text{ for all }t \in \mathbb{R}
    \end{equation*}
    and
    \begin{equation*}
  g(t)=\begin{cases}
    -1, & \text{if $t\leq-1$}.\\
    t, & \text{if $-1 \leq t \leq 1$}\\
    1, & \text{if $t \geq 1$}.
  \end{cases}
\end{equation*}
It is clear that $p_{1}(f-g) = \sup\big\{ |(f-g)(t)|:\; -1 \leq t \leq 1\big\} = 0$, that is, $\pi_{1}(f) = \pi_{1}(g)$ but $ \Phi(f) = 2 \neq 1 =  \Phi(g).$ Hence $\varphi_{1}\left( \pi_{1}(f)\right) = \Phi(f)$ is not well-defined. Equvalently, $\Phi \neq \varphi_{1}\circ \pi_{1}.$
\end{eg}
We draw the functions $f$ and $g$ defined in the Example \ref{eg: Phi not associated to varphi1} here. We point out that any two continuous functions coincide on the interval $[-1, 1]$ and differ at the point $2$ will serve as an example. 
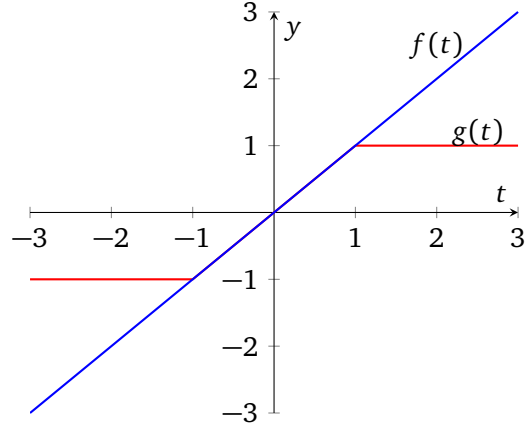
\begin{equation*}
    \begin{tikzpicture}[
  declare function={
    func(\x)= (\x <= -1) * (-1)   +
              and(\x >=-1, \x <= 1) * (x)     +
              (\x >= 1) * (1)
   ;
  }
]
\begin{axis}[
  axis x line=middle, axis y line=middle,
  ymin=-3, ymax=3, ytick={-3,...,3}, ylabel=$y$,
  xmin=-3, xmax=3, xtick={-3,...,3}, xlabel=$t$,
  domain=-3:3,samples=100, 
]
\addplot [red,thick] {func(x)};
\addplot[blue,thick]{\x};
\node at (2.5,1.2) {$g(t)$};
\node at (2,2.5) {$f(t)$};
\end{axis}
\end{tikzpicture} 
\end{equation*}
\\

 Now we are in a situation to define the character space of a commutative unital locally $C^{\ast}$-algebra. In view of Theorem \ref{thm: Phi associated to varphialpha} we consider the multiplicative linear functional that are induced from the multiplicative linear functional on the quotient $C^{\ast}$-algebra. We give the formal definition below. 

\begin{defn}\label{defn: characterspace}
    Let $\mathcal{A}$ be a commutative unital locally $C^{\ast}$-algebra and let $(\Lambda, \leq)$ be a directed POSET. For each $\alpha \in \Lambda,$ the quotient algebra $\mathcal{A}_{\alpha}$ is a commutative unital $C^{\ast}$-algebra and $\pi_{\alpha}\colon \mathcal{A} \to \mathcal{A}_{\alpha}$ is a continuous cononical quotient map. We denote the character space of $\mathcal{A}$ by $\mathcal{M}_{\mathcal{A}}$ and define as,
    \begin{equation*}
     \mathcal{M}_\mathcal{A}=\Big\{\Phi \colon \mathcal{A} \rightarrow \mathbb{C}\; \text{is multiplicative linear map}\;:\; \Phi=\varphi_{\alpha} \circ \pi_\alpha, \text{ for some } \varphi_{\alpha} \in \mathcal{M}_{\mathcal{A}_{\alpha}},\; \alpha \in \Lambda\Big\}.
 \end{equation*}
\end{defn}  
It is evident from Definition \ref{defn: characterspace} that $\mathcal{M}_{\mathcal{A}}$ is non-empty set since each $\mathcal{M}_{\mathcal{A}_{\alpha}}$ is non-empty, for $\alpha \in \Lambda$. Moreover, if $\Phi \in \mathcal{M}_{\mathcal{A}}$ then $\Phi$ is a unital map. Since $\mathcal{M}_{\mathcal{A}} \subseteq \mathcal{A}^{\ast}$, the class of all continuous linear functionals on $\mathcal{A}$, one can consider $\mathcal{M}_{\mathcal{A}}$ is equipped with weak$^{\ast}$-topology. However, unlike the case of $C^{\ast}$-algebras, $\mathcal{M}_{\mathcal{A}}$ may not be weak$^{\ast}$-compact (see Example \ref{eg: Not compact}).

\begin{eg}\label{eg: Not compact}
    Consider the commutative unital locally $C^{\ast}$-algebra $\mathcal{A}=C(\mathbb{R}).$ It is shown in Example \ref{Example: loc C* alg} that 
    \begin{equation*}
         C(\mathbb{R})= \varprojlim\limits_{n \in \mathbb{N}} C\left([-n, n]\right).
    \end{equation*}
    By Eberlein-Smulian theorem, $\mathcal{M}_{C(\mathbb{R})}$ is weak$^{\ast}$-compact if and only if every sequence in $\mathcal{M}_{C(\mathbb{R})}$ has a convergent subsequence. Let $\nu_{x}$ denotes the evaluation functional at $x \in \mathbb{R}$ given by
    \begin{equation*}
        \nu_{x}(f)=f(x), \text{ for all }f \in C(\mathbb{R}).
    \end{equation*}
    Now consider the sequence $\{n\}_{n \in \mathbb{N}}$  and a continuous function $g \colon \mathbb{R} \rightarrow \mathbb{R}$ given by $g(x)=x,$ for all $x \in \mathbb{R}.$ Since the sequence $\{\nu_n(g)=n\}_{n \in \mathbb{N}}$ is not cauchy in $\mathbb{R},$ it follows that the sequence $\{\nu_n\}_{n \in \mathbb{N}}$ of evaluation functionals does not have a weak$^{\ast}$-convergent subsequence and consequently $\mathcal{M}_{C(\mathbb{R})}$ is not weak$^{\ast}$-compact. 
    
\end{eg}
    
    
    

\subsection{$\mathcal{M}_{\mathcal{A}}$ is the inductive limit of $\{\mathcal{M}_{\mathcal{A}_\alpha}\}_{\alpha \in \Lambda}$}\label{subsection: inductive limit of maximal ideal spaces}

Next, we explore the connection between our notion of character space $\mathcal{M}_{\mathcal{A}}$ (see Definition \ref{defn: characterspace}) and the well known maximal ideal space of a commutative $C^{\ast}$-algebra $\mathcal{A}_\alpha$, where $\mathcal{A}_{\alpha} = \mathcal{A}/\mathcal{I}_{\alpha}$, for $\alpha \in \Lambda$ (see Remark \ref{rmk: Invlimit}). Let us recall that the map $\pi_{\alpha, \beta}\colon \mathcal{A}_{\beta} \to \mathcal{A}_{\alpha}$ defined by $\pi_{\alpha, \beta}(a+\mathcal{I}_{\beta}) = a + \mathcal{I}_{\alpha}$ is surjective $C^{\ast}$-homomorphism (whenever $\alpha \leq \beta$) and $\left( \big\{\mathcal{A}_{\alpha}\big\}_{\alpha \in \Lambda},\; \{\pi_{\alpha, \beta}\}_{\alpha \leq \beta}\right)$ is projective system such that $$\mathcal{A} = \varprojlim\limits_{\alpha \in \Lambda} \mathcal{A}_{\alpha}.$$

\noindent Firstly, assume that $\alpha \leq \beta.$ If $\varphi_\alpha \in \mathcal{M}_{\mathcal{A}_\alpha},$ 
then for $a,b \in \mathcal{A}$ we see that 
\begin{align*}
    \varphi_{\alpha} \circ \pi_{\alpha, \beta} \left( (a+\mathcal{I}_{\beta}) \cdot (b+\mathcal{I}_{\beta}) \right) &= \varphi_{\alpha} \circ \pi_{\alpha, \beta}\big(a\cdot b+\mathcal{I}_{\beta}\big) \\
    &= \varphi_\alpha\big(a\cdot b+\mathcal{I}_{\alpha}\big) \\
    &=\varphi_\alpha \big( (a+\mathcal{I}_{\alpha}) \cdot (b+\mathcal{I}_{\alpha}) \big) \\
    &= \varphi_\alpha\big(a+\mathcal{I}_{\alpha}\big) \cdot \varphi_\alpha\big(b+\mathcal{I}_{\alpha}\big) \\
    &= \varphi_{\alpha} \circ \pi_{\alpha, \beta}\big(a+\mathcal{I}_{\beta}\big) \cdot \varphi_{\alpha} \circ \pi_{\alpha, \beta}\big(b+\mathcal{I}_{\beta}\big).  
\end{align*}
That is, $\varphi_{\alpha} \circ \pi_{\alpha, \beta} \in \mathcal{M}_{\mathcal{A}_\beta}.$ In view of this, whenever $\alpha \leq \beta$, there exists a natural map $\up{\up{\gamma}}_{\beta,\alpha} \colon \mathcal{M}_{\mathcal{A}_\alpha} \rightarrow \mathcal{M}_{\mathcal{A}_\beta}$ defined by
\begin{equation*}
  \up{\up{\gamma}}_{\beta,\alpha}(\varphi_\alpha) = \varphi_\alpha \circ \pi_{\alpha,\beta}, \text{ for all } \varphi_\alpha \in \mathcal{M}_{\mathcal{A}_\alpha}.
\end{equation*}
To show that $\up{\up{\gamma}}_{\beta, \alpha}$ is continuous, let us consider an open set $V = \bigcap\limits_{i=1}^{N} \nu_{\pi_{\beta}(x_{i})}^{-1}(U)$ in $\mathcal{A}_{\beta}$ for some open $U$ in $\mathbb{C},$ where $\nu_{\pi_\beta(x_i)}$ is evaluation map and $x_{1}, x_{2}, \cdots, x_{N} \in \mathcal{A}$ then 
\begin{align*}
    \up{\up{\gamma}}_{\beta, \alpha}^{-1}(V) &=\Big\{\varphi_\alpha \in \mathcal{M}_{\mathcal{A}_\alpha} \colon \varphi_\alpha \circ \pi_{\alpha,\beta} \in  \bigcap\limits_{i=1}^{N} \nu_{\pi_{\beta}(x_{i})}^{-1}(U)\Big\} \\
   &= \bigcap^{N}_{i=1}\Big\{\varphi_\alpha \in \mathcal{M}_{\mathcal{A}_\alpha} \colon \varphi_\alpha \circ \pi_{\alpha,\beta} \in \nu^{-1}_{\pi_{\beta}(x_{i})}(U)\Big\} \\
    &=\bigcap^{n}_{i=1}\Big\{\varphi_\alpha \in \mathcal{M}_{\mathcal{A}_\alpha} \colon \varphi_\alpha \circ \pi_{\alpha,\beta}(\pi_{\beta}(x_{i})) \in U\Big\} \\
    &=\bigcap^{n}_{i=1}\Big\{\varphi_\alpha \in \mathcal{M}_{\mathcal{A}_\alpha} \colon \varphi_\alpha \circ \pi_{\alpha}(x_{i})\in U\Big\}\\
    &= \bigcap\limits_{i=1}^{N} \nu_{\pi_{\alpha}(x_{i})}^{-1}(U),
\end{align*}
which is open in $\mathcal{M}_{\mathcal{A}_\alpha}$. This implies that $\up{\up{\gamma}}_{\beta, \alpha}$ is continuous whenever $\alpha \leq \beta$. Also $\up{\gamma}_{\alpha,\alpha}$ is the identity map on $\mathcal{M}_{\mathcal{A}_\alpha}.$ Moreover, if $\alpha \leq \beta \leq \delta$ then 
\begin{equation} \label{eq: gamma transitivity}
     \up{\up{\gamma}}_{\delta,\beta} \circ \up{\up{\gamma}}_{\beta,\alpha}(\varphi_\alpha)  = \up{\up{\gamma}}_{\delta,\beta}(\varphi_\alpha \circ \pi_{\alpha,\beta}) 
    = \varphi_\alpha \circ \pi_{\alpha,\beta} \circ \pi_{\beta,\delta} 
     = \varphi_\alpha \circ \pi_{\alpha,\delta} 
     = \up{\up{\gamma}}_{\delta,\alpha}(\varphi_\alpha), 
\end{equation}
for all $\varphi_\alpha \in \mathcal{M}_{\mathcal{A}_\alpha}.$ By adopting the notion given in Definition \ref{defn: ind system}, we say that $\left(\{\mathcal{M}_{\mathcal{A}_\alpha}\}_{\alpha \in \Lambda}, \{\up{\gamma}_{\beta,\alpha}\}_{\alpha \leq \beta} \right)$ is an inductive system of topological spaces (more precisely, weakly compact spaces). 

In order to compute the inductive limit of inductive system $\left(\{\mathcal{M}_{\mathcal{A}_\alpha}\}_{\alpha \in \Lambda}, \{\up{\gamma}_{\beta,\alpha}\}_{\alpha \leq \beta}\right)$, we construct a corresponding strict inductive system.  
For each $\alpha \in \Lambda,$ define
\begin{equation*}
    \mathcal{Z}_\alpha := \big\{\varphi_\alpha \circ \pi_\alpha \colon \varphi_\alpha \in \mathcal{M}_{\mathcal{A}_\alpha}\big\}.
\end{equation*}
It is the collection of all those multiplicative linear functionals $\mathcal{A}$ induced from $\mathcal{M}_{\mathcal{A}_\alpha}.$ We shall show that $\big\{\mathcal{Z}_\alpha\big\}_{\alpha \in \Lambda}$ is a strict inductive system. To see this, let us assume that $\alpha \leq \beta$ and $\varphi_{\alpha} \in \mathcal{M}_{\mathcal{A}_{\alpha}}$ then from previous observation, we know that $\psi_{\beta}:= \varphi_{\alpha} \circ \pi_{\alpha, \beta} \in \mathcal{M}_{\mathcal{A}_{\beta}}$. By using the fact that $\pi_{\alpha, \beta} \circ \pi_{\beta} = \pi_{\alpha}$, we get  
\begin{equation*}
    \varphi_{\alpha}\circ \pi_{\alpha} = \varphi_{\alpha}\circ\left( \pi_{\alpha, \beta} \circ \pi_{\beta}\right) = \psi_{\beta} \circ \pi_{\beta} \in \mathcal{Z}_{\beta}. 
\end{equation*}
It follows that $\mathcal{Z}_{\alpha} \subseteq \mathcal{Z}_{\beta}$ whenever $\alpha \leq \beta$.  Now we define a topology on $\mathcal{Z}_{\alpha}.$ For $\alpha \in \Lambda$, let $\Omega_{\alpha}$ be the collection of all subsets $S \subseteq \mathcal{Z}_{\alpha}$ of the form $S = \big\{ \varphi_{\alpha}\circ \pi_{\alpha}:\; \varphi_{\alpha} \in \mathcal{O}_{\alpha}\big\}$ for some open set $\mathcal{O}_{\alpha}$ in $\mathcal{M}_{\mathcal{A}_{\alpha}}.$ Then we see that
\begin{enumerate}
    \item both $\emptyset,\; \mathcal{Z}_{\alpha} \in \Omega_{\alpha}$;
    \item if $S_{i} = \big\{\varphi_{\alpha}\circ \pi_{\alpha}:\; \varphi_{\alpha} \in \mathcal{O}_{\alpha, i} \big\}$ for some $\mathcal{O}_{\alpha, i}$ open in $\mathcal{M}_{\mathcal{A}_{\alpha}},\; i \in I \;(\text{indexing set}),$ then 
    \begin{equation*}
        \bigcup\limits_{i \in I} S_{i} = \bigcup\limits_{i \in I}\{\varphi_\alpha \circ \pi_\alpha \colon \varphi_\alpha \in \mathcal{O}_{\alpha, i}\} =\{\varphi_\alpha \circ \pi_\alpha \colon \varphi_\alpha \in \bigcup\limits_{i \in I}\mathcal{O}_{\alpha, i}\} \in \Omega_{\alpha};
    \end{equation*}
    \item if $S_{i} = \big\{\varphi_{\alpha}\circ \pi_{\alpha}:\; \varphi_{\alpha} \in \mathcal{O}_{\alpha, i} \big\}$ for some $\mathcal{O}_{\alpha, i}$ open in $\mathcal{M}_{\mathcal{A}_{\alpha}},\; i=1,2,\cdots,n,$ then 
    \begin{equation*}
        \bigcap\limits_{i=1}^{n} S_{i} = \bigcap\limits_{i=1}^{n}\{\varphi_\alpha \circ \pi_\alpha \colon \varphi_\alpha \in \mathcal{O}_{\alpha, i}\} =\{\varphi_\alpha \circ \pi_\alpha \colon \varphi_\alpha \in \bigcap\limits_{i=1}^{n}\mathcal{O}_{\alpha, i}\} \in \Omega_{\alpha}.
    \end{equation*}
    \end{enumerate}

\noindent Therefore, $\big\{\big(\mathcal{Z}_\alpha,\,\Omega_\alpha\big)\big\}_{\alpha \in \Lambda}$ is a strict inductive system of topological spaces. Moreover, 

\begin{equation} \label{Equation: Strictinductivelimit}
    \mathcal{M}_{\mathcal{A}} = \bigcup\limits_{\alpha \in \Lambda}\mathcal{Z}_\alpha.
\end{equation}
Here $\mathcal{M}_{\mathcal{A}}$ is equipped with (strict) inductive limit topology, that is the strongest topology on $\mathcal{M}_{\mathcal{A}}$ under which each inclusion map $i_{\alpha}\colon \mathcal{Z}_\alpha \hookrightarrow \mathcal{M}_{\mathcal{A}}$ is continuous, for all $\alpha \in \Lambda.$ Now we give an explicit description of the {\it inductive limit topology} on $\mathcal{M}_\mathcal{A}.$
 Let us define $\tau$ as follows, 
\begin{equation}\label{eq: inductive limit topology}
    \tau = \{S \subseteq \mathcal{M}_\mathcal{A} \colon S \, \cap \, \mathcal{Z}_\alpha \text{ is open in }\mathcal{Z}_\alpha, \text{ for all } \alpha \}.
\end{equation}
Firstly, we show that $\tau$ is a topology on $\mathcal{M}_\mathcal{A}$.
\begin{itemize}
    \item[(a)] Clearly, $\emptyset, \mathcal{M}_{\mathcal{A}} \in \tau.$
    \item[(b)] If $\{V_i\}_{i \in I} \subseteq \tau$ for some arbitrary index set $I$, then  $V_i \cap \mathcal{Z}_\alpha$ is open in $\mathcal{Z}_\alpha,$ for all $\alpha\ \in \Lambda, i \in I.$ It follows that 
    \begin{equation*}
        \left( \bigcup\limits_{i\in I} V_i \right) \cap \mathcal{Z}_\alpha = \bigcup\limits_{i \in I}(V_i \cap \mathcal{Z}_\alpha)
    \end{equation*} 
    is open in $\mathcal{Z}_\alpha$, for each $\alpha \in \Lambda.$ That is, $\bigcup\limits_{i\in I}V_{i} \in \tau.$ 
    \item[(c)] If $\{V_k\}_{k=1}^{n} \subseteq \tau$  then  $V_k \cap \mathcal{Z}_\alpha$ is open in $\mathcal{Z}_\alpha,$ for all $\alpha\ \in \Lambda$ and $ k=1,2, \cdots, n.$ This implies that 
    \begin{equation*}
        \left( \bigcap\limits_{k=1}^{n} V_{k} \right) \cap \mathcal{Z}_{\alpha} = \bigcap\limits_{k=1}^{n} V_{k} \cap \mathcal{Z}_{\alpha}
    \end{equation*}
    is open in $\mathcal{Z}_{\alpha},$ for each $\alpha \in \Lambda.$ Hence $\displaystyle\bigcap^n_{k=1}V_k \in \tau.$ \newline
\end{itemize}
Therefore, $\tau$ is a topology on $\mathcal{M}_\mathcal{A}$. Let $\alpha \in \Lambda$ and for any $S \in \tau$, then   $i^{-1}_\alpha(S) = S \cap \mathcal{Z}_\alpha$ is open in $\mathcal{Z}_\alpha$. As a result, every inclusion map $i_{\alpha}\colon \mathcal{Z}_{\alpha} \to \mathcal{M}_{\mathcal{A}}$ is continuous. 
Further, we show that $\tau$ is the strongest topology under which each inclusion map is continuous. Suppose $\tau^{\prime}$ is a topology on $\mathcal{M}_{\mathcal{A}}$ under which each $i_{\alpha}$ is continuous. If $S \in \tau^{\prime}$ then for every $\alpha \in \Lambda$ we have 
\begin{equation*}
    S \cap \mathcal{Z}_{\alpha} = i_{\alpha}^{-1}(S)\; \text{is open in}\; \mathcal{Z}_{\alpha}.
\end{equation*}
Equivalently, $S \in \tau.$ This shows that $\tau^{\prime} \subseteq \tau. $

Now we establish a concrete relation between the character space of commutative unital locally $C^{\ast}$-algebra $\mathcal{A}$ and the inductive system of maximal ideal spaces of corresponding quotient $C^{\ast}$-algebras. In order to show that $\mathcal{M}_{\mathcal{A}}$ is an inductive limit of $(\{\mathcal{M}_{\mathcal{A}_\alpha}\}_{\alpha \in \Lambda}, \{{\up{\gamma}}_{\beta,\alpha}\}_{\alpha \leq \beta})$, for each $\alpha \in \Lambda$ we define the map ${\up{\gamma}}_\alpha : \mathcal{M}_{\mathcal{A}_\alpha} \rightarrow \mathcal{M}_\mathcal{A}$ by 
\begin{equation*}
  {\up{\gamma}}_\alpha(\varphi_\alpha) = \varphi_\alpha \circ \pi_\alpha, \text{ for all } \varphi_\alpha \in \mathcal{M}_{\mathcal{A}_\alpha}. 
\end{equation*}
We prove that ${\up{\gamma}}_{\alpha}$ is continuous. Note that $\mathcal{M}_{\mathcal{A}_{\alpha}}$ is equipped with the weak$^{\ast}$-topology. Let $\alpha \in \Lambda$ be fixed. If $S \in \tau$ then $S \cap \mathcal{Z}_\alpha$ is open in $\mathcal{Z}_\alpha$, that is, there is an open set $\mathcal{O}_{\alpha}$ in $\mathcal{M}_{\mathcal{A}_{\alpha}}$  such that $S \cap \mathcal{Z}_{\alpha} = \big\{ \psi_{\alpha} \circ \pi_{\alpha}:\; \psi_{\alpha} \in \mathcal{O}_{\alpha}\big\}.$ It follows that
\begin{align*}
\varphi_{\alpha} \in \up{\gamma}_{\alpha}^{-1}(S) \; \text{if and only if}\; \varphi_{\alpha} \circ \pi_{\alpha} \in S \cap \mathcal{Z}_{\alpha} \; \text{if and only if}\; \varphi_{\alpha} \in \mathcal{O}_{\alpha}.    
\end{align*}
This shows that $\up{\gamma}_{\alpha}^{-1}(S)$ is open in $\mathcal{M}_{\mathcal{A}_{\alpha}}$ for any $S \in \tau.$ Therefore, $\up{\gamma}_\alpha$ is continuous for every $\alpha \in \Lambda.$ Whenever $\alpha \leq \beta$, we get
\begin{equation} \label{eq: compatible}
     \left(\up{\gamma}_\beta \circ \up{\gamma}_{\beta,\alpha}\right)(\varphi_\alpha)  = \up{\gamma}_\beta(\varphi_\alpha \circ \pi_{\alpha,\beta}) 
     = \left(\varphi_\alpha \circ \pi_{\alpha,\beta}\right) \circ \pi_\beta 
     = \varphi_\alpha \circ \pi_\alpha 
     = \up{\gamma}_\alpha(\varphi_\alpha), 
\end{equation}
for every $\varphi_\alpha \in \mathcal{M}_{\mathcal{A}_\alpha}$. As a result, $\up{\gamma}_{\beta} \circ \up{\gamma}_{\beta,\alpha}= \up{\gamma}_{\alpha}$ whenever $\alpha \leq \beta.$ Since $\mathcal{M}_{\mathcal{A}}$ is a topological space and $\{\up{\gamma}_\alpha\}_{\alpha \in \Lambda}$ is a family of continuous maps satisfying Equation \eqref{eq: compatible}, we see that $\left(\mathcal{M}_\mathcal{A}, \{\up{\gamma}_\alpha\}_{\alpha \in \Lambda}\right)$ is compatible with the inductive system $\left(\{\mathcal{M}_{\mathcal{A}_\alpha}\}_{\alpha \in \Lambda}, \{\up{\gamma}_{\beta,\alpha}\}_{\alpha \leq \beta}\right).$ Therefore, 
\begin{equation}\label{Eq: M_A inductive limit} 
   \mathcal{M}_\mathcal{A} = \varinjlim_{\alpha \in \Lambda}\mathcal{M}_{\mathcal{A}_\alpha}.
\end{equation}
\begin{rmk}\label{rmk: continuity and homeomorphism}
    There is a comparison between inductive limit topology and weak$^{\ast}$-topology on $\mathcal{M}_\mathcal{A}.$
    \begin{enumerate}
        \item The inductive limit topology on $\mathcal{M}_{\mathcal{A}}$ is finer than the weak$^{\ast}$-topology on $\mathcal{M}_{\mathcal{A}}$. This is because, if $\bigcap\limits^{N}_{i=1}\nu^{-1}_{x_i}(U)$ is open in $\mathcal{M}_{\mathcal{A}},$ for some $x_1,x_2,\cdots,x_N \in \mathcal{A}$ and some open set $U$ in $\mathbb{C},$ then 
        \begin{align*}
            \up{\gamma}^{-1}_\alpha\Big(\bigcap^{N}_{i=1}\nu^{-1}_{x_i}(U)\Big) &=\big\{\varphi_\alpha \in \mathcal{M}_{\mathcal{A}_\alpha} \colon \varphi_\alpha \circ \pi_\alpha \in \bigcap^{N}_{i=1}\nu^{-1}_{x_i}(U)\big\} \\
    &=\bigcap^{N}_{i=1}\big\{\varphi_\alpha \in \mathcal{M}_{\mathcal{A}_\alpha} \colon \varphi_\alpha \circ \pi_\alpha \in \nu^{-1}_{x_i}(U)\big\} \\
    &= \bigcap^{N}_{i=1}\big\{\varphi_\alpha \in \mathcal{M}_{\mathcal{A}_\alpha} \colon \varphi_\alpha \circ \pi_\alpha(x_i) \in U\big\} \\
    &= \bigcap^{N}_{i=1}\nu^{-1}_{\pi_\alpha(x_i)}(U),
        \end{align*}
    which is weak$^{\ast}$-open in $\mathcal{M}_{\mathcal{A}_\alpha}.$ This is, $\up{\gamma}_{\alpha}$ is continuous for every $\alpha \in \Lambda.$ 

    \item $\mathcal{M}_{\mathcal{A}_\alpha}$ is homeomorphic to $\mathcal{Z}_\alpha.$ Consider the continuous map $i_\alpha^{-1} \circ \up{\gamma}_\alpha \colon \mathcal{M}_{\mathcal{A}_\alpha} \to \mathcal{Z}_\alpha,$ where $i_\alpha^{-1} \circ \up{\gamma}_\alpha(\varphi_\alpha)=\varphi_\alpha \circ \pi_\alpha,$ for all $\varphi_\alpha \in \mathcal{M}_{\mathcal{A}_\alpha}.$ Note that, $i_\alpha^{-1} \circ \up{\gamma}_\alpha$ is bijective. To see this, consider $\varphi_\alpha,\;\psi_\alpha$ in $\mathcal{M}_{\mathcal{A}_\alpha}$ such that $i_\alpha^{-1} \circ \up{\gamma}_\alpha(\varphi_\alpha)=i_\alpha^{-1} \circ \up{\gamma}_\alpha(\psi_\alpha).$ Then
    \begin{align*}
        \varphi_\alpha \circ \pi_\alpha=\psi_\alpha \circ \pi_\alpha &\implies \varphi_\alpha \circ \pi_\alpha(a)=\psi_\alpha \circ \pi_\alpha(a),\text{ for all } a \in \mathcal{A} \\
        &\implies \varphi_\alpha(a+\mathcal{I}_\alpha)=\psi_\alpha(a+\mathcal{I}_\alpha),\text{ for all } a \in \mathcal{A} \\
        &\implies \varphi_\alpha=\psi_\alpha.
    \end{align*}
    Therefore $i_\alpha^{-1} \circ \up{\gamma}_\alpha$ is injective. Moreover by the definition of  $i_\alpha^{-1} \circ \up{\gamma}_\alpha,$ it is surjective, hence bijective. Further, for any open set $S=\left\{\varphi_\alpha \circ \pi_\alpha \colon \varphi_\alpha \in \mathcal{O}_\alpha\right\}$ in $\mathcal{Z}_\alpha,$ where $\mathcal{O}_\alpha$ is open in $\mathcal{M}_{\mathcal{A}_\alpha},$
    \begin{equation*}
        (i_\alpha^{-1} \circ \up{\gamma}_\alpha)^{-1}(S)=\up{\gamma}_\alpha^{-1} \circ i_\alpha(S)=\up{\gamma}^{-1}_{\alpha}\left(\{\varphi_\alpha \circ \pi_\alpha \colon \varphi_\alpha \in \mathcal{O}_\alpha\}\right)=\mathcal{O}_\alpha,
    \end{equation*}
    which is weak$^{\ast}$-open in $\mathcal{M}_{\mathcal{A}_\alpha}.$ Therefore,  $(i_\alpha^{-1} \circ \up{\gamma}_\alpha)^{-1}$ is continuous and consequently, $i_\alpha^{-1} \circ \up{\gamma}_\alpha$ is a homeomorphism.
    \end{enumerate}
\end{rmk}
\begin{prop}\label{prop: topological properties of character sp}
Let $\mathcal{A}$ be a commutative unital locally $C^{\ast}$-algebra. Then $\mathcal{M}_\mathcal{A}$ with inductive limit topology is completely regular. In particular, 
\begin{enumerate}
    \item if the underlying directed POSET $\Lambda = \mathbb{N}$, then $\mathcal{M}_\mathcal{A}$ is $\sigma$-compact.
    \item If $\mathcal{A}$ is a locally von-Nuemann algebra, then $\mathcal{M}_{\mathcal{A}}$ is the inductive limit of extremally disconnected spaces.
\end{enumerate}
\end{prop}
\begin{proof}
    We know from $(2)$ of Remark \ref{rmk: continuity and homeomorphism} that $\mathcal{M}_{\mathcal{A}_\alpha}$ is homeomorphic to $\mathcal{Z}_\alpha,$ for every $\alpha \in \Lambda.$ Since $\mathcal{M}_{\mathcal{A}_\alpha}$ is weak$^{\ast}$-compact and Hausdorff, it follows that $\mathcal{Z}_\alpha$ is compact and Hausdorff with respect to $\Omega_\alpha.$ Further, from Equation \eqref{Equation: Strictinductivelimit}, it follows that $\mathcal{M}_\mathcal{A}$ is locally compact and Hausdorff (being strict inductive limit of Hausdorff spaces). Therefore $\mathcal{M}_{\mathcal{A}}$ possesses a one point compactification, say $\mathring{\mathcal{M}}_{\mathcal{A}}$ and $\mathring{\mathcal{M}}_{\mathcal{A}}$ is compact Hausdorff space, hence normal. This implies that $\mathring{\mathcal{M}}_{\mathcal{A}}$ is completely regular and so $\mathcal{M}_\mathcal{A}$. \\

    \noindent Proof of $(1)$: Suppose $\Lambda=\mathbb{N},$ then $\mathcal{M}_\mathcal{A}=\bigcup\limits_{n \in \mathbb{N}}\mathcal{Z}_n,$ the countable union of compact sets and hence $\mathcal{M}_\mathcal{A}$ is $\sigma$-compact.\\

   \noindent Proof of $(2)$: In particular if $\mathcal{A}$ is commutative unital locally von Neumann algebra, then each $\mathcal{A}_\alpha$ is commutative unital von Neumann algebra. By Theorem 9.6 of \cite{kehezhu}, the maximal ideal space $\mathcal{M}_{\mathcal{A}_\alpha}$ is extremally disconnected space (in fact it is a Stonean space) for each $\alpha \in \Lambda.$ So $\mathcal{Z}_\alpha$ is extremally disconnected  from $(2)$ of Remark \ref{rmk: continuity and homeomorphism}. It follows that, $\mathcal{M}_{\mathcal{A}}$ is the (strict) inductive limit of extremally disconnected spaces. 
\end{proof}
\begin{note}\label{note: relation between character space and maximal ideals}
Now we point out a few observations in regard to maximal ideals in a commutative unital locally $C^{\ast}$-algebra $\mathcal{A}$ and give a  relation between the character space of $\mathcal{A}$ and certain type of maximal ideals. Firstly note that, if $\Phi \in \mathcal{M}_\mathcal{A},$ then $\Phi=\varphi_\beta \circ \pi_\beta,$ for some $\beta \in \Lambda$ and $\varphi_\beta \in \mathcal{M}_{\mathcal{A}_\beta}.$ It follows that
\begin{align*}
    J := Ker(\Phi) &= \big\{x \in \mathcal{A} \colon \Phi(x)=0\big\} \\
    &= \big\{x \in \mathcal{A} \colon \varphi_\beta(\pi_\beta(x))=0\big\} \\
    &= \big\{x \in \mathcal{A} \colon \pi_\beta(x) \in Ker(\varphi_\beta)\big\}.
\end{align*}
It is immediate to see that, $\mathcal{I}_\beta \subseteq J$ and $\pi_\beta(J)=Ker(\varphi_\beta)$ is a maximal ideal in $\mathcal{A}_\beta$ (follows from \cite[Theorem 4.3]{kehezhu}). Consequently, $J$ is a maximal ideal. This can be seen as follows: suppose $J$ is not a maximal ideal, then there is a proper ideal $J^{\prime}$ of $\mathcal{A}$ such that $J \subsetneq J^{\prime} \subsetneq \mathcal{A}$. There exists an $x \in J^{\prime}$ with $x \notin J$ and so, $\pi_{\alpha}(x) \in \pi_{\alpha}(J^{\prime}) \setminus \pi_{\alpha}(J)$. Since $\mathcal{I}_{\alpha} \subseteq J$, we see that $\pi_{\alpha}(J^{\prime})$ is a proper ideal in $\mathcal{A}_{\alpha}$ such that $\pi_{\alpha}(J) \subsetneq \pi_{\alpha}(J^{\prime}) \subsetneq \mathcal{A}_{\alpha}.$ This is a contradiction to the fact that $\pi_{\alpha}(J)$ is maximal in $\mathcal{A}_{\alpha}.$
\end{note}
From the above observation, the maximal ideals of $\mathcal{A}$ containing $\mathcal{I}_{\alpha}$ for some $\alpha \in \Lambda$ seem to have a connection with the members of $\mathcal{M}_\mathcal{A}$. The following result gives one to one correspondence between them. 
\begin{thm}\label{thm: correspondence}
    If $\mathcal{A}$ is a commutative unital locally $C^{\ast}$-algebra, then $\mathcal{M}_\mathcal{A}$ is in one to one correspondence with the maximal ideals in $\mathcal{A}$ containing $\mathcal{I}_\alpha,$ 
    \,for some $\alpha \in \Lambda.$ 
\end{thm}
\begin{proof}
    Firstly, take $\Phi=\varphi_\beta \circ \pi_\beta \in \mathcal{M}_{\mathcal{A}}$ for some $\beta \in \Lambda,$ then by Note \ref{note: relation between character space and maximal ideals}\; $\mathcal{I}_\beta \subseteq Ker(\Phi)$ and $Ker(\Phi)$ is a maximal ideal. Conversely, if $J$ is a maximal ideal in $\mathcal{A}$ containing $\mathcal{I}_\alpha,$ for some $\alpha \in \Lambda.$ We show that $\pi_\alpha(J)$ is maximal ideal in $\mathcal{A}_\alpha.$ If possible assume that $\pi_\alpha(J)$ is not maximal in $\mathcal{A}_\alpha,$ then there is a proper ideal $K^{(\alpha)}$ in $\mathcal{A}_{\alpha}$ with 
    $\pi_{\alpha}(J) \subsetneq K^{(\alpha)} \subsetneq \mathcal{A}_{\alpha}.$ This implies that $\pi_{\alpha}^{-1}(K^{(\alpha)})$ is a proper ideal in $\mathcal{A}$ satisfying, 
    \begin{equation*}
        J \subsetneq \pi_{\alpha}^{-1}(K^{(\alpha)}) \subsetneq \mathcal{A}.
    \end{equation*}
    This is a contradiction to the fact that $J$ is a maximal ideal in $\mathcal{A}.$ Therefore $\pi_\alpha(J)$ is maximal in $\mathcal{A}_\alpha.$ By Theorem $4.3$ of \cite{kehezhu}, there exists $\psi_\alpha \in \mathcal{M}_{\mathcal{A}_\alpha}$ such that $Ker(\psi_\alpha)=\pi_\alpha(J).$ It follows that the map $\Psi=\psi_\alpha \circ \pi_\alpha \in \mathcal{M}_\mathcal{A}$ and since $\mathcal{I}_{\alpha} \subseteq J,$ we have
    \begin{align*}
        Ker(\Psi)&=\big\{x \in \mathcal{A} \colon\; \pi_\alpha(x) \in Ker(\psi_\alpha)\big\} = \big\{x \in \mathcal{A} \colon x+\mathcal{I}_\alpha \in \pi_\alpha(J)\big\} = J.
    \end{align*}
    We conclude that the map $\Phi \mapsto Ker(\Phi)$ defines one-to-one correspondence between the space $\mathcal{M}_{\mathcal{A}}$ and the collection of all maximal ideals in $\mathcal{A}$ containing $\mathcal{I}_{\alpha}$ for some $\alpha \in \Lambda.$
\end{proof}
\section{Gelfand type representation of Locally C* Algebras}\label{section: Gelfand type}
 In this section, we establish a \textit{Gelfand type representation} of commutative unital locally $C^{\ast}$-algebra $\mathcal{A}.$ Recall that if $(\Lambda, \leq)$ is the underlying directed POSET, then $\mathcal{A} = \varprojlim\limits_{\alpha \in \Lambda} \mathcal{A}_{\alpha}$, where  $\mathcal{A}_{\alpha}:= \mathcal{A}/ \mathcal{I}_{\alpha}$ is a commutative unital $C^{\ast}$-algebra and $\big(\{\mathcal{A}_{\alpha}\}_{\alpha \in \Lambda},\{\pi_{\alpha,\beta}\}_{\alpha \leq \beta}\big)$ is a projective system (see Definition \ref{defn: proj system}). From the well-known Gelfand representation theorem\cite{kehezhu} of commutative unital $C^{\ast}$-algebra, for each $\alpha \in \Lambda$, there is an isometric $\ast$-isomorphism ($C^{\ast}$- representation) ${\Gamma}_\alpha \colon \mathcal{A}_{\alpha} \to C(\mathcal{M}_{\mathcal{A}_{\alpha}})$ given by 
 \begin{equation}\label{eq: Gelfand map}
     \Gamma_\alpha(x_\alpha)(\varphi_\alpha)=\varphi_\alpha(x_\alpha), \text{ for all } x_\alpha \in \mathcal{A}_\alpha \,, \varphi_\alpha \in \mathcal{M}_{\mathcal{A}_\alpha}.
 \end{equation}
 Note that $\Gamma_{\alpha}(x_{\alpha})$ is a continuous function on $\mathcal{M}_{\mathcal{A}_{\alpha}}$ for each $\alpha \in \Lambda.$
 \begin{thm}\label{Thm: C(M_A) is local C* alg}
     For a commutative unital locally $C^{\ast}$-algebra $\mathcal{A},$ we have 
     \begin{equation*}
         C(\mathcal{M}_{\mathcal{A}})=\varprojlim\limits_{\alpha \in \Lambda}C(\mathcal{M}_{\mathcal{A}_\alpha}).
     \end{equation*}
     In other words, $C(\mathcal{M}_\mathcal{A})$ is a commutative unital locally $C^{\ast}$-algebra.
 \end{thm}
 \begin{proof}
 Let us consider the family $\Big\{ C(\mathcal{M}_{\mathcal{A}_\alpha})\Big\}_{\alpha \in \Lambda}$ of commutative unital $C^{\ast}$-algebras and whenever $\alpha \leq \beta,$ define $\xi_{\alpha,\beta} : C(\mathcal{M}_{\mathcal{A}_\beta}) \rightarrow C(\mathcal{M}_{\mathcal{A}_\alpha})$ by
 

\begin{equation*}
    \xi_{\alpha,\beta}(f_\beta) = f_\beta \circ \up{\gamma}_{\beta,\alpha} \text{ for all } f_\beta \in C(\mathcal{M}_{\mathcal{A}_\beta}),  
\end{equation*}
where $\up{\gamma}_{\beta,\alpha}(\varphi_\alpha)=\varphi_\alpha \circ \pi_{\alpha,\beta}$ is a continuous map from $\mathcal{M}_{\mathcal{A}_\alpha}$ to $\mathcal{M}_{\mathcal{A}_\beta}$ 
(see Section \ref{subsection: inductive limit of maximal ideal spaces}).
It is easy to see that, for $\alpha \leq \beta,$ $\xi_{\alpha,\beta}$ is a $C^{\ast}$-homomorphism. The surjectivity of $\xi_{\alpha, \beta}$ can be seen as follows: suppose $f_{\alpha} \in C(\mathcal{M}_{\mathcal{A}_{\alpha}})$ then the map $g_{\alpha} \colon \mathcal{Z}_{\alpha} \to \mathbb{C}$ given by 
\begin{equation*}
    g_{\alpha}(\varphi_{\alpha}\circ \pi_{\alpha}) = f_{\alpha}(\varphi_{\alpha}),\; \text{for all}\; \varphi_{\alpha} \in \mathcal{M}_{\mathcal{A}_{\alpha}}.
\end{equation*}
is a well-defined and continuous. Since $\mathcal{Z}_{\beta}$ is normal $\big($being compact Hausdorff space with respect to $\Omega_{\alpha}\big)$ and $\mathcal{Z}_{\alpha}$ is compact in $\mathcal{Z}_{\beta}$, by Tietze-extension theorem, there is a continuous map $g_{\beta} \colon \mathcal{Z}_{\beta} \to \mathbb{C}$ such that 
\begin{equation*}
    g_{\beta}(\varphi_{\alpha} \circ \pi_{\alpha}) = g_{\alpha}(\varphi_{\alpha} \circ \pi_\alpha),\; \text{for every}\; \varphi_{\alpha} \in \mathcal{M}_{\mathcal{A}_{\alpha}}.
\end{equation*}
Now define the map $f_{\beta} \colon \mathcal{M}_{\mathcal{A}_{\beta}} \to \mathbb{C}$ as,
\begin{equation*}
    f_{\beta}(\varphi_{\beta}) = g_{\beta}(\varphi_{\beta} \circ \pi_{\beta}), \; \text{for all}\; \varphi_{\beta} \in \mathcal{M}_{\mathcal{A}_{\beta}}.
\end{equation*}
It follows that $f_{\beta} \in C(\mathcal{M}_{\mathcal{A}_{\beta}}).$ Moreover, if $\varphi_{\alpha} \in \mathcal{M}_{\mathcal{A}_{\alpha}}$, then $\varphi_{\alpha} \circ \pi_{\alpha, \beta} \in \mathcal{M}_{\mathcal{A}_{\beta}}$ and we get \begin{align*}
    \xi_{\alpha, \beta}(f_{\beta}) (\varphi_{\alpha}) = (f_{\beta} \circ \up{\gamma}_{\beta, \alpha})(\varphi_{\alpha}) &= f_{\beta} \left( \varphi_{\alpha}\circ \pi_{\alpha, \beta}\right)\\
    &= g_{\beta}\left( (\varphi_{\alpha}\circ \pi_{\alpha, \beta}) \circ \pi_{\beta}\right)\\
    &= g_{\beta}\left( \varphi_{\alpha}\circ \pi_{\alpha}\right)\\
    &= g_{\alpha}\left( \varphi_{\alpha}\circ \pi_{\alpha}\right)\\
    &= f_{\alpha}(\varphi_{\alpha}).
\end{align*}
As $\varphi_{\alpha} \in  \mathcal{M}_{\mathcal{A}_{\alpha}}$ is chosen arbitrarily, we conclude that $\xi_{\alpha, \beta}(f_{\beta}) = f_{\alpha}.$ Thus $\xi_{\alpha, \beta}$ is a surjective $C^{\ast}$-homomorphism, whenever $\alpha \leq \beta.$ The following diagram gives a clear description about the surjectivity of $\xi_{\alpha,\beta}$ via bridge maps.
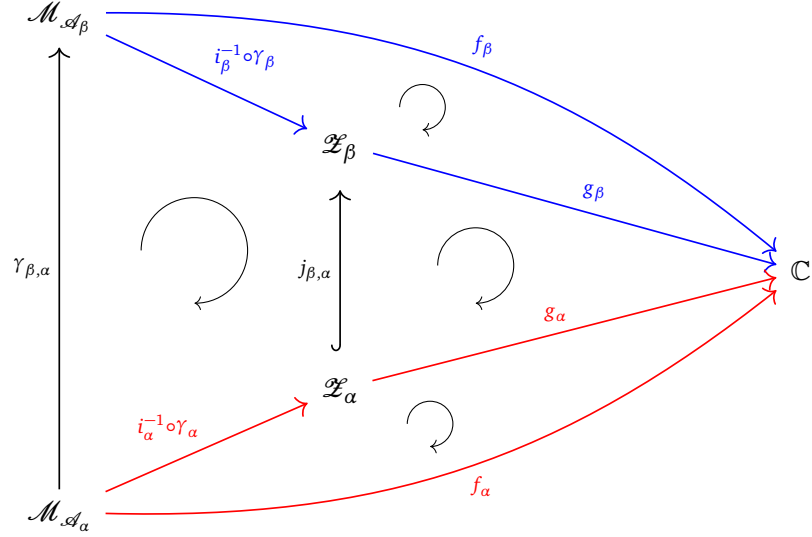
\begin{equation*}
\begin{tikzcd}[sep=large, every matrix/.append style={name=m},
    execute at end picture={
    \draw [<-] ([xshift=-2.5cm,yshift=1mm]m-2-2.west) arc[start angle=-90,delta angle=270,radius=0.7cm];
    \draw [<-] ([xshift=1.2cm,yshift=1mm]m-2-2.west) arc[start angle=-90,delta angle=270,radius=0.5cm]; 
    \draw [<-] ([xshift=0.5cm,yshift=2.4cm]m-2-2.west) arc[start angle=-90,delta angle=270,radius=3mm];
    \draw [<-] ([xshift=0.6cm,yshift=-1.8cm]m-2-2.west) arc[start angle=-90,delta angle=270,radius=3mm]; }]
    \mathcal{M}_{\mathcal{A}_\beta} \ar[blue,drr,"i_{\beta}^{-1} \circ \up{\gamma}_\beta"] \ar[blue,ddrrrrrr, bend left=20,"f_{\beta}"] &&&&&& \\
    && \mathcal{Z}_\beta \ar[blue,drrrr,"g_{\beta}"] &&&&\\
    &&&&&& \mathbb{C} \\
    && \mathcal{Z}_\alpha \ar[red,urrrr,"g_{\alpha}"] \ar[uu,hook,"j_{\beta,\alpha}",shorten = 2mm] &&&& \\
    \mathcal{M}_{\mathcal{A}_\alpha} \ar[uuuu,"\up{\gamma}_{\beta,\alpha}"] \ar[red,urr,"i_{\alpha}^{-1} \circ \up{\gamma}_\alpha"] \ar[red, uurrrrrr,swap,bend right=20,"f_{\alpha}"] &&&&&&
\end{tikzcd}
\end{equation*}

Further, whenever $\alpha \leq \beta\leq \delta$ and $f_{\delta} \in C(\mathcal{M}_{\mathcal{A}_\delta})$, by using Equation \eqref{eq: gamma transitivity} we see that
\begin{equation*}
    \left( \xi_{\alpha,\beta} \circ \xi_{\beta,\delta}\right)(f_\delta) = \xi_{\alpha,\beta}(f_\delta \circ \up{\gamma}_{\delta,\beta}) = \left(f_\delta \circ \up{\gamma}_{\delta,\beta}\right) \circ \up{\gamma}_{\beta,\alpha} = f_\delta \circ \up{\gamma}_{\delta,\alpha} = \xi_{\alpha, \delta}(f_{\delta}).
\end{equation*}
This shows that $\xi_{\alpha,\beta} \circ \xi_{\beta,\delta} = \xi_{\alpha,\delta}$ whenever  $\alpha \leq \beta \leq \delta$ and hence  $\left(\{C(\mathcal{M}_{\mathcal{A}_\alpha})\}_{\alpha \in \Lambda},\{\xi_{\alpha,\beta}\}_{\alpha \leq \beta}\right)$ is a projective system of $C^{\ast}$-algebras. 

On the other hand, for each $\alpha \in \Lambda$, we define $q_\alpha(f)=\|f \circ \up{\gamma}_\alpha\|_{\infty},$ where $\up{\gamma}_\alpha \colon \mathcal{M}_{\mathcal{A}_\alpha} \rightarrow \mathcal{M}_{\mathcal{A}}$ is a continuous map (see Subsection \ref{subsection: inductive limit of maximal ideal spaces}) and $f \in C(\mathcal{M}_\mathcal{A}).$ Note that $f \circ \up{\gamma}_\alpha \in C(\mathcal{M}_{\mathcal{A}_\alpha})$ and $\mathcal{M}_{\mathcal{A}_\alpha}$ is compact, so $f \circ \up{\gamma}_\alpha$ is bounded. Whenever $\alpha \leq \beta$ and $f \in C(\mathcal{M}_{\mathcal{A}})$, we have 
\begin{align*}
    q_{\alpha}(f) = \|f \circ \gamma_\alpha\|_{\infty} = \sup \Big\{|f \circ \gamma_\alpha(\varphi_\alpha)| \colon \varphi_\alpha \in \mathcal{M}_{\mathcal{A}_\alpha}\Big\} 
    &= \sup \Big\{|f(\varphi_\alpha \circ \pi_\alpha)| \colon \varphi_\alpha \in \mathcal{M}_{\mathcal{A}_\alpha}\Big\} \\
    &= \sup \Big\{\left|f\left( (\varphi_\alpha \circ \pi_{\alpha,\beta}) \circ \pi_\beta \right)\right| \colon \varphi_\alpha \in \mathcal{M}_{\mathcal{A}_\alpha}\Big\}.
    \end{align*}
    Since $\varphi_{\alpha} \circ \pi_{\alpha, \beta} \in \mathcal{M}_{\mathcal{A}_{\beta}}$ for every $\varphi_{\alpha} \in \mathcal{M}_{\mathcal{A}_{\alpha}}$, it follows that $q_{\alpha}(f) \leq q_{\beta}(f).$ Also, we get 
\begin{equation*}
    q_{\alpha}(f^{\ast}f) = q_{\alpha}(|f|^{2}) = \left\||f|^{2} \circ \up{\gamma}_{\alpha}\right\|_{\infty} = \sup\Big\{ |f(\varphi_{\alpha}\circ \pi_{\alpha})|^{2}:\; \varphi_{\alpha} \in \mathcal{M}_{\mathcal{A}_{\alpha}}\Big\} = q_{\alpha}(f)^{2}.
\end{equation*}
Thus $\big\{q_\alpha\big\}_{\alpha \in \Lambda}$ is an upward filtered family of $C^{\ast}$-seminorms on $C(\mathcal{M}_{\mathcal{A}})$ .  As a result, we consider the unital $\ast$-algebra $C(\mathcal{M}_{\mathcal{A}})$ equipped with the locally convex topology induced by the family $\big\{q_\alpha\big\}_{\alpha \in \Lambda}$. Finally, for each $\alpha \in \Lambda$, let $\xi_\alpha : C(\mathcal{M}_{\mathcal{A}}) \rightarrow C(\mathcal{M}_{\mathcal{A}_\alpha})$ be defined by 
\begin{equation*}
    \xi_\alpha(f) = f \circ \up{\gamma}_\alpha, \text{ for all } f \in C(\mathcal{M}_{\mathcal{A}}).
\end{equation*}
To show that $\xi_{\alpha}$ is continuous, let $U$ be open in $C(\mathcal{M}_{\mathcal{A}_\alpha})$ given by
 $   U=\left\{g \in C(\mathcal{M}_{\mathcal{A}_\alpha}) \colon \|g\|_{\infty} < r\right\},$ for some $r>0.$
Then,
\begin{align*}
    \xi^{-1}_{\alpha}(U)&=\left\{f \in C(\mathcal{M}_{\mathcal{A}}) \colon \|f \circ \up{\gamma}_\alpha\|_{\infty} < r\right\} \\
    &= \left\{f \in C(\mathcal{M}_{\mathcal{A}}) \colon q_\alpha(f) < r\right\} \\
    &= q^{-1}_{\alpha}(B(0,r)),
\end{align*}
where $B(0,r)$ is an open ball in $\mathbb{C}$ of radius $r$ centered at $0,$ is open in locally convex topology in $C(\mathcal{M}_{\mathcal{A}}).$ Thus, $\xi_\alpha$ is continuous. Further, we have 
\begin{equation*}
   \xi_{\alpha,\beta} \circ \xi_\beta(f)  = \xi_{\alpha,\beta}(f \circ \up{\gamma}_\beta) 
    = f \circ \up{\gamma}_\beta \circ \up{\gamma}_{\beta,\alpha} 
    = f \circ \up{\gamma}_\alpha 
    = \xi_\alpha(f), 
\end{equation*}
for all $f \in C(\mathcal{M}_{\mathcal{A}})$ whenever $\alpha \leq \beta.$ This shows that $\big(C(\mathcal{M}_{\mathcal{A}}),\,\{\xi_\alpha\}_{\alpha \in \Lambda}\big)$ is compatible with the projective system $\big(\{C(\mathcal{M}_{\mathcal{A}_\alpha})\}_{\alpha \in \Lambda}, \{\xi_{\alpha,\beta}\}_{\alpha \leq \beta}\big)$ and hence 
\begin{equation*}
    C(\mathcal{M}_{\mathcal{A}})=\varprojlim\limits_{\alpha \in \Lambda}C(\mathcal{M}_{\mathcal{A}_\alpha}).
\end{equation*}
\end{proof}
\begin{rmk}\label{rmk: identification net}
    Every $f \in C(\mathcal{M}_\mathcal{A})$ can be identified with the net $\big\{f \circ \gamma_\alpha\big\}_{\alpha \in \Lambda} \in \prod\limits_{\alpha \in \Lambda}C(\mathcal{M}_{\mathcal{A}_\alpha}).$ Indeed the net $\big\{f \circ \gamma_\alpha\big\}_{\alpha \in \Lambda}$ is such that $\xi_{\alpha,\beta}(f \circ \gamma_\beta)=f \circ \gamma_\beta \circ \gamma_{\beta,\alpha}=f \circ \gamma_\alpha,$ whenever $\alpha \leq \beta.$ 
\end{rmk}
\subsection{Gelfand type representation}\label{subsection: Gelfand type representation}
Recall that $\mathcal{A}$ is a commutative unital locally $C^{\ast}$-algebra and for each $\alpha \in \Lambda$, the Gelfand representation of $C^{\ast}$- algebra $\mathcal{A}_{\alpha}$, denoted by $\Gamma_{\alpha}$ (see Equation \eqref{eq: Gelfand map}). We give a Gelfand type representation for $\mathcal{A}$ and establish a connection with the family $\{\Gamma_{\alpha}\}_{\alpha \in \Lambda}.$ Since the character space $\mathcal{M}_{\mathcal{A}}$ consists of all those multiplicative linear functional of the form $\varphi_{\alpha} \circ \pi_{\alpha}$ for some $\alpha \in \Lambda$ and $\varphi_{\alpha} \in \mathcal{M}_{\mathcal{A}_{\alpha}}$, we define $\Gamma \colon \mathcal{A} \rightarrow C(\mathcal{M}_{\mathcal{A}})$ by 
\begin{equation} \label{eq: Gelfand type representation}
    \Gamma(a)(\varphi_{\alpha}\circ \pi_{\alpha})=\varphi_\alpha \circ \pi_\alpha(a)=\varphi_\alpha(a+\mathcal{I}_\alpha), \text{ for all } a \in \mathcal{A}.
\end{equation}
In the following theorem we show that $\Gamma$ is a coherent representation of locally $C^{\ast}$-algebras. 
\begin{thm}\label{Thm: gelfand coherent}
    The map $\Gamma$ defined in Equation \eqref{eq: Gelfand type representation} is a coherent representation of locally $C^{\ast}$-algebras $\mathcal{A}$ and $C(\mathcal{M}_\mathcal{A}).$ In fact, $\Gamma$ is a unital local contractive (isometric) $\ast$- homomorphism. 
\end{thm}
\begin{proof}
    We know from Section \ref{section: loc maximal ideal sp} that $\big(\mathcal{A},\{\pi_{\alpha}\}_{\alpha \in \Lambda}\big)$ is a projective limit of $\big(\{\mathcal{A}_\alpha\}_{\alpha \in \Lambda},\; \{\pi_{\alpha,\beta}\}_{\alpha \leq \beta}\}\big)$ and as shown in Theorem \ref{Thm: C(M_A) is local C* alg}\; $\big(C(\mathcal{M}_{\mathcal{A}}),\; \{\xi_\alpha\}_{\alpha \in \Lambda}\big)$ is a projective limit of $\big(\{C(\mathcal{M}_{\mathcal{A}_\alpha})\}_{\alpha \in \Lambda},\; \{\xi_{\alpha,\beta}\}_{\alpha \leq \beta}\big).$ Firstly note that, for a fixed $\alpha \in \Lambda,\; a \in \mathcal{A},\; \varphi_\alpha \in \mathcal{M}_{\mathcal{A}_\alpha},$ we see that
    \begin{align*}
        \xi_\alpha\big(\Gamma(a)\big)(\varphi_\alpha)=\Gamma(a)\big( \gamma_\alpha(\varphi_\alpha)\big)=\Gamma(a)(\varphi_\alpha \circ \pi_\alpha)&=(\varphi_\alpha \circ \pi_\alpha)(a)\\
        &=\varphi_\alpha(a+\mathcal{I}_\alpha) \\
        &=\Gamma_\alpha(a+\mathcal{I}_\alpha)(\varphi_\alpha) \\
        &=\Gamma_\alpha\left(\pi_\alpha(a)\right)(\varphi_\alpha) \\
        &= \big(\Gamma_\alpha \circ \pi_\alpha\big)(a)(\varphi_\alpha).
    \end{align*}
    Since $a \in \mathcal{A}$ and $\varphi_\alpha \in \mathcal{M}_{\mathcal{A}_\alpha}$ are chosen arbitrarily, we conclude that 
    \begin{equation*}
        \xi_\alpha \circ \Gamma= \Gamma_\alpha \circ \pi_\alpha, \text{ for every } \alpha \in \Lambda.
    \end{equation*}
    Here we have shown that, there exists a net $\big\{\Gamma_\alpha\big\}_{\alpha \in \Lambda}$ of $C^{\ast}$-representations $\Gamma_\alpha \colon \mathcal{A}_\alpha \to C(\mathcal{M}_{\mathcal{A}_\alpha}),\; \alpha \in \Lambda$ such that the following diagram commutes:
    \begin{equation*}
    \begin{tikzcd}[every matrix/.append style={name=m},
    execute at end picture={
    \draw [<-] ([xshift=4mm,yshift=3mm]m-2-2.west) arc[start angle=-90,delta angle=270,radius=0.33cm]; }] 
    \mathcal{A}  \arrow[thick,swap] {dd}{\Gamma} \arrow[thick]{rr}{\pi_\alpha} && \mathcal{A}_\alpha \arrow[thick]{dd}{\Gamma_\alpha} \\
     && \\
    C(\mathcal{M}_{\mathcal{A}})  \arrow[thick,swap]{rr}{\xi_\alpha} && C(\mathcal{M}_{\mathcal{A}_\alpha})   
    \end{tikzcd}
\end{equation*}
Equivalently, whenever $\alpha \leq \beta,\; \varphi_\alpha \in \mathcal{M}_{\mathcal{A}_\alpha}$ and $a \in \mathcal{A},$ we have
\begin{align*}
    \xi_{\alpha,\beta}\big( \Gamma_\beta(a+\mathcal{I}_\beta)\big)(\varphi_\alpha) &= \Gamma_\beta(a+\mathcal{I}_\beta) \circ \gamma_{\beta,\alpha}(\varphi_\alpha) \\
    &= \Gamma_\beta(a+\mathcal{I}_\beta)(\varphi_\alpha \circ \pi_{\alpha,\beta}) \\
    &= \varphi_\alpha \circ \pi_{\alpha,\beta}(a+\mathcal{I}_\beta) \\
    &= \varphi_\alpha(\pi_{\alpha,\beta}(a+\mathcal{I}_\beta)) \\
    &=\big(\Gamma_\alpha \circ \pi_{\alpha,\beta}\big)(a+\mathcal{I}_\beta)(\varphi_\alpha).
\end{align*}
Since $a \in \mathcal{A}$ and $\varphi_\alpha \in \mathcal{M}_{\mathcal{A}_\alpha}$ are chosen arbitrarily, we get that 
    \begin{equation*}
        \xi_{\alpha,\beta} \circ \Gamma_\beta=\Gamma_\alpha \circ \pi_{\alpha,\beta}, \text{ whenever } \alpha \leq \beta.
    \end{equation*}
    Hence $\Gamma$ is a coherent representation of locally $C^{\ast}$-algebras. Now we show that $\Gamma$ is a unital local contractive (isometric) $\ast$-homomorphism. If $1_{\mathcal{A}} \in \mathcal{A}$ is a unit, then 
    \begin{equation*}
       \Gamma(1_{\mathcal{A}})(\varphi_\alpha \circ \pi_\alpha)=\varphi_\alpha(\pi_\alpha(1_\mathcal{A}))=\varphi_\alpha(1_{\mathcal{A}_\alpha})=1
    \end{equation*} 
    for every $a \in \mathcal{A}$ and $\varphi_\alpha \in \mathcal{M}_{\mathcal{A}_\alpha},$ so $\Gamma$ is unital. Clearly, $\Gamma$ is a homomorphism. For $a \in \mathcal{A}$ and $\varphi_{\alpha} \in \mathcal{M}_{\mathcal{A}_{\alpha}},\; \alpha \in \Lambda$ we have
    \begin{equation*}
    \Gamma(a^{\ast})(\varphi_\alpha \circ \pi_\alpha)=\varphi_\alpha\big(a^{\ast}+\mathcal{I}_\alpha\big)=\varphi_\alpha\big((a+\mathcal{I}_\alpha)^{\ast}\big)=\varphi_\alpha\big(a+\mathcal{I}_\alpha\big)^{\ast}=\Gamma(a)^{\ast}(\varphi_\alpha \circ \pi_\alpha).
\end{equation*}
Thus $\Gamma(a^{\ast})=\Gamma(a)^{\ast}.$ Further, by using the fact that $\Gamma_{\alpha}$ is an isometry for each $\alpha \in \Lambda$, we get
\begin{align*}
    q_\alpha(\Gamma(a))=\|\Gamma(a)\circ \up{\gamma}_\alpha\|_{\infty}  &= \sup \Big\{ |\left(\Gamma(a)\circ \up{\gamma}_\alpha \right) (\varphi_{\alpha})|:\; \varphi_{\alpha} \in \mathcal{M}_{\mathcal{A}_{\alpha}}\Big\} \\
    &= \sup \Big\{| \Gamma(a)\left(\varphi_{\alpha}\circ \pi_{\alpha} \right)| :\; \varphi_{\alpha} \in \mathcal{M}_{\mathcal{A}_{\alpha}}\Big\} \\
     &= \sup \Big\{ |\varphi_{\alpha}(\pi_{\alpha}(a))| :\; \varphi_{\alpha} \in \mathcal{M}_{\mathcal{A}_{\alpha}}\Big\} \\
     &= \sup \Big\{ |\Gamma_{\alpha}\left( \pi_{\alpha}(a)\right)(\varphi_{\alpha})| :\; \varphi_{\alpha} \in \mathcal{M}_{\mathcal{A}_{\alpha}}\Big\} \\
     &= \left\|\Gamma_{\alpha}\left( \pi_{\alpha}(a)\right) \right\|_{\infty}\\
     &= \|\pi_{\alpha}(a)\|\\
     &=p_{\alpha}(a),
\end{align*}
for every $\alpha \in \Lambda.$ Hence $\Gamma$ is a local contractive (isometric) $\ast$-homomorphism. 
\end{proof}
\begin{note}
    An appeal to Remark \ref{rmk: identification net}, we see that for every $a \in \mathcal{A},$ the continuous function $\Gamma(a)$ on $\mathcal{M}_\mathcal{A}$ is identified with the net $\big\{\Gamma(a) \circ \gamma_\alpha\big\}_{\alpha \in \Lambda}.$ However, it follows from Theorem \ref{Thm: gelfand coherent} that $\Gamma(a) \circ \gamma_\alpha=\Gamma_\alpha \circ \pi_\alpha(a),$ for each $\alpha \in \Lambda.$ As a result, $\Gamma$ can be identified with $\{\Gamma_\alpha\}_{\alpha \in \Lambda}$ in the sense that
    \begin{equation*}
        \xi_\alpha\big(\Gamma(a)\big)=\Gamma_\alpha\big(\pi_\alpha(a)\big), \text{ for all }a \in \mathcal{A}.
    \end{equation*}
\end{note}
\subsection{The Continuous Functional Calculus}\label{section: cont. functional calculus}
In this subsection, we define the functional calculus of locally bounded normal operator on a locally Hilbert space. Recall that if $\mathcal{E}:=\{\mathcal{H}_{\alpha}\}_{\alpha \in \Lambda}$ is an upward filtered family of Hilbert spaces and $\mathcal{D}:= \varinjlim\limits_{\alpha \in \Lambda} \mathcal{H}_{\alpha} = \bigcup\limits_{\alpha \in \Lambda}\mathcal{H}_{\alpha}$ is a locally Hilbert space, then $C^{\ast}_{\mathcal{E}}(\mathcal{D})$ is a unital locally $C^{\ast}$-algebra \big(see (1) of Definition \ref{defn: locally von neumann alg}\big).  Firstly, we discuss the spectrum of locally bounded operators. Let $T \in C^{\ast}_{\mathcal{E}}(\mathcal{D}).$ Then $T$ is locally bounded, that is, each $\mathcal{H}_{\alpha}$ is reducing under $T$ and $T_{\alpha}:= T\big|_{\mathcal{H}_{\alpha}} \in \mathcal{B}(\mathcal{H}_{\alpha})$ for every $\alpha \in \Lambda$. The \textit{local resolvent} of $T$ is defined as, 
\begin{equation*}
    \rho_{\it loc}(T) := \Big\{\lambda \in \mathbb{C} \colon \;  \left(\lambda \cdot I_{\mathcal{D}}-T \right) \text{ is invertible in}\; C^{\ast}_{\mathcal{E}}(\mathcal{D})\Big\},
\end{equation*}
where $I_{\mathcal{D}}$ is the identity operator on $\mathcal{D}.$ In other words, $\lambda \in \rho_{loc}(T)$ if and only if there is a unique (locally bounded operator)\; $S \in C_{\mathcal{E}}^{\ast}(\mathcal{D})$ such that 
\begin{equation} \label{eq: invertible}
    S\left(\lambda I_{\mathcal{D}}-T\right) = \left(\lambda I_{\mathcal{D}}-T\right) S = I_{\mathcal{D}}. 
\end{equation}  
The {\it local spectrum} of $T$ is denoted by $\sigma_{\it loc}(T)$ and we define this as, 
\begin{equation}\label{eq: local spectrum}
    \sigma_{ loc}(T) = \mathbb{C}\setminus \rho_{\textit loc}(T).
\end{equation}
\begin{thm}\label{Theorem: local spectrum}
    Let $T \in C^{\ast}_{\mathcal{E}}(\mathcal{D})$. Then $\sigma_{loc}(T)$ is a non-empty subset of $\mathbb{C}$. Moreover, 
    \begin{equation*}
        \sigma_{loc}(T) = \bigcup\limits_{\alpha \in \Lambda} \sigma(T_{\alpha}),
    \end{equation*}
    where $T_{\alpha}:= T\big|_{\mathcal{H}_{\alpha}}$ and $\sigma(T_{\alpha})$ is the spectrum of the bounded operator $T_{\alpha}.$
\end{thm}
\begin{proof}
    Let $\lambda \in \mathbb{C}.$ Then $\lambda \in \rho_{loc}(T)$ if and only if there is a unique  $S \in C^{\ast}_{\mathcal{E}}(\mathcal{D})$ satisfying Equation \eqref{eq: invertible}. Since $S$ is locally bounded, we know that $S_{\alpha}:= S\big|_{\mathcal{H}_{\alpha}} \in \mathcal{B}(\mathcal{H}_{\alpha})$ and 
    \begin{equation*}
        S_{\alpha} (\lambda I_{\mathcal{H}_{\alpha}} - T_{\alpha}) =  (\lambda I_{\mathcal{H}_{\alpha}} - T_{\alpha}) S_{\alpha} = I_{\mathcal{H}_{\alpha}}, \; \text{for every}\; \alpha \in \Lambda. 
    \end{equation*}
    That is, $\lambda \in \rho(T_{\alpha})$ for every $\alpha \in \Lambda.$ We have shown that $\rho_{loc}(T) = \bigcap\limits_{\alpha \in \Lambda} \rho(T_{\alpha}).$ Equivalently, 
    \begin{equation*}
        \sigma_{loc}(T) = \bigcup\limits_{\alpha \in \Lambda}\sigma(T_{\alpha}).
    \end{equation*}
    Since $\sigma(T_{\alpha})$ is a non-empty compact subset of $\mathbb{C}$, it follows that $\sigma_{loc}(T)$ is non-emtpy.  
\end{proof}
\begin{rmk} A locally bounded operator $T \colon \mathcal{D} \to \mathcal{D}$ can be seen as a densely defined (unbounded) operator, not necessarily closed. 
    Suppose $\mathcal{H}$ is the completion of $\mathcal{D}$. The (unbounded) spectrum of $T \in C^{\ast}_{\mathcal{E}}(\mathcal{D})$ is given by 
    \begin{equation*}
        \sigma(T) = \big\{ \lambda \in \mathbb{C}:\; (\lambda I_{\mathcal{D}}-T)\; \text{does not have a bounded inverse}\big\}.
    \end{equation*}
    The reason we call the notion of spectrum of locally bounded operator $T$ as ``local spectrum'' and denoted by $\sigma_{loc}(T)$ in order to point out that it is different from $\sigma(T).$ The same is described in the following example. 
\end{rmk}
\begin{eg} \label{eg: number matrix}
    Let us consider the locally Hilbert space $\mathcal{D} = \bigcup\limits_{n\in \mathbb{N}} \mathcal{H}_{n}$, where $\mathcal{H}_{n}:= span\{e_{1}, e_{2}, \cdots, e_{n}\}$ subspace of $\ell^{2}(\mathbb{N})$ \big(as in Example \ref{example: loc hilbert eg} \big) for each $n \in \mathbb{N}$. Now define $T : \mathcal{D} \rightarrow \mathcal{D}$ by 
\[
T = \begin{pmatrix}
1 & 0 & 0 & 0 & 0 & \cdots \\
0 & \frac{1}{2} & 0 & 0 & 0 & \cdots \\
0 & 0 & 3 & 0 & 0 & \cdots \\
0 & 0 & 0 & \frac{1}{4} & 0 & \cdots \\
0 & 0 & 0 & 0 & 5 & \cdots \\
\vdots & \vdots & \vdots & \vdots & \vdots & \ddots
\end{pmatrix}
\]
That is, $T(x_{1}, x_{2}, x_{3}, x_{4}, \cdots) = \left(x_{1}, \frac{x_{2}}{2}, 3 x_{3}, \frac{x_{4}}{4}, \cdots \right)$ for every $\{x_{n}\}_{n\in \mathbb{N}} \in \mathcal{D}.$ It follows that each  $\mathcal{H}_{n}$ is reducing under $T$ and $T_n=T\big|_{\mathcal{H}_n} \in \mathcal{B}(\mathcal{H}_{n})$ and so, $T \in C^{\ast}_{\mathcal{E}}(\mathcal{D}).$ Also, $T$ is an unbounded operator. Now we show that $T$ is not closed. Let us consider the sequence $\{X_n\}_{n \in \mathbb{N}}$ in $\mathcal{D},$ where 
\begin{equation*}
    X_{1} = e_{1} \;\; \text{and}\;\;  X_{n} = e_{1} + \sum\limits_{k=1}^{n-1} \frac{1}{2k} e_{2k},\;\; \text{for}\; n \geq 1.
\end{equation*}
If $ X := e_{1} + \sum\limits_{k = 1}^{\infty} \frac{1}{2k}e_{2k},\; Y:= e_{1} + \sum\limits_{k = 1}^{\infty} \frac{1}{4k^2}e_{2k}$, then $X,Y \in \ell^{2}(\mathbb{N})$ such that  
\begin{equation*}
   \big\{ X_{n}\big\}_{n \in \mathbb{N}} \longrightarrow X  \; \text{and} \; \big\{T(X_{n})\big\}_{n\in \mathbb{N}} = \left\{ e_{1} + \sum\limits_{k=1}^{n-1} \frac{1}{4k^2} e_{2k}\right\}_{n\in \mathbb{N}} \longrightarrow Y,
\end{equation*}
as $n \to \infty,$ but $X \notin \mathcal{D}.$ This shows that $T$ is not a closed operator. 

Firstly, note that each $T_{n}$ is a finite rank operator. So, the spectrum of $T_{n}$ is computed as,
\[
  \sigma(T_n)=\left\{ \begin{array}{cc}
    \big\{1,\frac{1}{2},3\cdots 2m-1\big\}, \;& \text{if $n=2m-1$ is odd};\\
     & \\
    \big\{1,\frac{1}{2},3\cdots 2m-1,\frac{1}{2m}\big\}, & \text{if $n=2m$ is even}.
  \end{array} \right.
\]
Therefore, the local spectrum of $T$ is,
\begin{equation*}
    \sigma_{loc}(T) = \Big\{ 1, \frac{1}{2}, 3, \frac{1}{4}, 5, \cdots\Big\}.
\end{equation*}
On the other hand, the (unbounded) spectrum of $T$, is given by $\sigma(T) = \big\{2k-1, \frac{1}{2k} : k \in \mathbb{N}\big\} \cup \{0\}$. Indeed, $0$ is in \textit{approximate point spectrum} of $T$. Since $0 \notin \sigma_{loc}(T)$, we conclude that  $\sigma_{loc}(T) \subsetneq \sigma(T).$ 
\end{eg}
\begin{note}
Suppose that  $T \in C^{\ast}_{\mathcal{E}}(\mathcal{D})$ is a normal operator, then it follows that  $T_\alpha=T\big|_{\mathcal{H}_\alpha}$ is a bounded normal operator on $\mathcal{H}_{\alpha}$ for every $\alpha \in \Lambda.$ Now, for each $\alpha \in \Lambda$, consider $C^{\ast}\{I_{\mathcal{H}_{\alpha}},\;T_{\alpha},T^{\ast}_\alpha\},$ the $C^{\ast}$-algebra generated by $T_\alpha.$  Since $T_\alpha$ is normal, $C^{\ast}\{I_{\mathcal{H}_{\alpha}},\;T_{\alpha},T^{\ast}_\alpha\}$ is a commutative unital $C^{\ast}$-algebra and the class of polynomials $p(T_{\alpha}, T^{\ast}_{\alpha})$ is dense in $C^{\ast}\{I_{\mathcal{H}_{\alpha}}, T_{\alpha}, T_{\alpha}^{\ast}\}$ by \cite[Proposition 10.1]{kehezhu} . Whenever $\alpha \leq \beta$, define  $\phi_{\alpha, \beta}\big( p(T_{\beta}, T_{\beta}^{\ast})\big) = p(T_{\alpha}, T_{\alpha}^{\ast})$  for every polynomial $p(T_{\beta}, T_{\beta}^{\ast}) \in C^{\ast}\{I_{\mathcal{H}_{\beta}}T_{\beta}, T_{\beta}^{\ast}\}$. By using the fact that $T_{\beta}\big|_{\mathcal{H}_{\alpha}} = T_{\alpha}$, we see that  $\phi_{\alpha, \beta}$ is surjective and 
\begin{align*}
    \|\phi_{\alpha, \beta}\big( p(T_{\beta}, T_{\beta}^{\ast})\big)\|_{\mathcal{H}_\alpha}=\| p(T_{\alpha}, T_{\alpha}^{\ast})\|_{\mathcal{H}_\alpha} &=\sup\left\{\|p(T_{\alpha}, T_{\alpha}^{\ast})(x)\|_{\mathcal{H}_\alpha} \colon \|x\|=1,\;x\in \mathcal{H}_\alpha\right\} \\
    &\leq \sup\left\{\|p(T_{\beta}, T_{\beta}^{\ast})(x)\|_{\mathcal{H}_\beta} \colon \|x\|=1,\;x\in \mathcal{H}_\beta\right\} \\
    &=\| p(T_{\beta}, T_{\beta}^{\ast})\|_{\mathcal{H}_\beta}.
\end{align*}
In particular, $\alpha \leq \beta \leq \gamma,$ we have
\begin{equation*}
    \phi_{\alpha, \beta} \circ \phi_{\beta, \gamma}\big(p(T_\gamma,T^{\ast}_\gamma)\big)=\phi_{\alpha, \beta}\big(p(T_\beta,T^{\ast}_\beta)\big)=p(T_\alpha,T^{\ast}_\alpha)=\phi_{\alpha, \gamma}\big(p(T_\gamma,T^{\ast}_\gamma)\big),
\end{equation*}
for all $p(T_\gamma,T^{\ast}_\gamma) \in C^{\ast}\{I_{\mathcal{H}_{\gamma}}, T_{\gamma},  T_{\gamma}^{\ast}\}.$
As a consequence, $\phi_{\alpha, \beta}$ has a unique continuous extension, again we denoted it by $\phi_{\alpha, \beta}$, from $C^{\ast}\{I_{\mathcal{H}_{\beta}}, T_{\beta}, T_{\beta}^{\ast}\}$ onto $C^{\ast}\{I_{\mathcal{H}_{\alpha}}, T_{\alpha}, T_{\alpha}^{\ast}\}$ such that
\begin{equation*}
    \phi_{\alpha, \beta} \circ \phi_{\beta, \gamma} = \phi_{\alpha, \gamma}, \; \text{whenever}\; \alpha \leq \beta \leq \gamma.
\end{equation*}
This follows that $\left(\big\{C^{\ast}\{I_{\mathcal{H}_{\alpha}}, T_{\alpha},  T_{\alpha}^{\ast}\}\big\}_{\alpha \in \Lambda},\; \big\{\phi_{\alpha, \beta}\big\}_{\alpha \leq \beta}\right)$ is a projective system of commutative unital $C^{\ast}$-algebras.  By the construction of projective limit given in Equation \eqref{Equation: Proj limit},
\begin{equation*}
    \varprojlim\limits_{\alpha \in \Lambda}C^{\ast}\{I_{\mathcal{H}_{\alpha}}, T_{\alpha},  T_{\alpha}^{\ast}\}=\left\{\big\{R_\alpha\big\}_{\alpha \in \Lambda} \in \prod\limits_{\alpha \in \Lambda}C^{\ast}\{I_{\mathcal{H}_{\alpha}}, T_{\alpha},  T_{\alpha}^{\ast}\} \colon \phi_{\alpha,\beta}(R_\beta)=R_\alpha \text{ whenever }\alpha \leq \beta\right\},
\end{equation*}
is a commutative unital locally $C^{\ast}$-algebra.
\end{note}
Now we define the locally $C^{\ast}$-algebra generated by given locally bounded operator.
\begin{defn}
    Let $T \in C^{\ast}_{\mathcal{E}}(\mathcal{D})$ be normal. The unital locally $C^{\ast}$-algebra generated by $T$ is denoted by $\mathcal{A}[T]$ and we define this as,
    \begin{equation}\label{eq: A[T]}
    \mathcal{A}[T]:=\left\{R=\varprojlim\limits_{\alpha \in \Lambda}R_\alpha \colon \big\{R_\gamma\big\}_{\gamma \in \Lambda} \in \varprojlim\limits_{\alpha \in \Lambda}C^{\ast}\{I_{\mathcal{H}_{\alpha}}, T_{\alpha},  T_{\alpha}^{\ast}\}\right\}.
\end{equation}
Here the symbol $\varprojlim\limits_{\alpha \in \Lambda}R_{\alpha}$ denotes a locally bounded normal operator on $\mathcal{D}$ whose restriction to $\mathcal{H}_\alpha$ is $R_\alpha,$ for every $\alpha \in \Lambda$ (see Remark \ref{rmk: locally bounded operator}).
\end{defn}
It is clear that $\mathcal{A}[T]$ is a subalgebra of $C^{\ast}_{\mathcal{E}}(\mathcal{D}).$ Now we show that $\mathcal{A}[T]$ is a commutative unital locally $C^{\ast}$-algebra. Firstly note that, the map $\phi_\alpha \colon \mathcal{A}[T] \rightarrow C^{\ast}\{I_{\mathcal{H}_{\alpha}}, T_{\alpha},  T_{\alpha}^{\ast}\}$ defined by 
\begin{equation*}
    \phi_\alpha(R)=R\big|_{\mathcal{H}_\alpha}, \text{ for all }R=\varprojlim\limits_{\alpha \in \Lambda}R_\alpha \in \mathcal{A}[T]
\end{equation*}
is a $\ast$-homomorphism. 
Further, $\mathcal{A}[T]$ together with the family $\{\phi_\alpha\}_{\alpha \in \Lambda}$ of $\ast$-homomorphisms is compatible with the projective system $\left(\big\{C^{\ast}\{I_{\mathcal{H}_{\alpha}}, T_{\alpha},  T_{\alpha}^{\ast}\}\big\}_{\alpha \in \Lambda},\; \big\{\phi_{\alpha, \beta}\big\}_{\alpha \leq \beta}\right)$ of $C^{\ast}$-algebras since
\begin{equation*}
    \phi_{\alpha,\beta} \circ \phi_\beta(R)=\phi_{\alpha,\beta}\big(R\big|_{\mathcal{H}_\beta}\big)=R\big|_{\mathcal{H}_\alpha}=\phi_{\alpha}(R),
\end{equation*}
for all $R \in \mathcal{A}[T],$ whenever $\alpha \leq \beta.$ It follows that $\mathcal{A}[T]$ is also the projective limit of the projective system $\left(\big\{C^{\ast}\{I_{\mathcal{H}_{\alpha}}, T_{\alpha},  T_{\alpha}^{\ast}\}\big\}_{\alpha \in \Lambda},\; \big\{\phi_{\alpha, \beta}\big\}_{\alpha \leq \beta}\right).$  Thus by Remark \ref{rmk: Invlimit}, $\mathcal{A}[T]$ is a commutative unital locally $C^{\ast}$-algebra.
In order to develop functional calculus for a locally bounded normal operator, we shall understand the relation between the character space $\mathcal{M}_{\mathcal{A}[T]}$ and the local spectrum $\sigma_{loc}(T).$ We give the result below. 
\begin{thm}\label{thm: homeomorphic to sigma_loc}
    Let $T \in C^{\ast}_\mathcal{E}(\mathcal{D})$ be normal. Then the character space $\mathcal{M}_{\mathcal{A}[T]}$ is homeomorphic to $\sigma_{loc}(T).$
\end{thm}
\begin{proof}
     Since $\mathcal{A}[T]$ is the projective limit of the projective system  $\left(\big\{C^{\ast}\{I_{\mathcal{H}_{\alpha}}, T_{\alpha},  T_{\alpha}^{\ast}\}\big\}_{\alpha \in \Lambda},\; \big\{\phi_{\alpha, \beta}\big\}_{\alpha \leq \beta}\right),$ consisting of locally bounded operators $R=\varprojlim\limits_{\alpha \in \Lambda}R_\alpha \in C^{\ast}_{\mathcal{E}}(\mathcal{D})$ satisfying $\phi_{\alpha,\beta}(R_\beta)=R_\alpha$ whenever $\alpha \leq \beta,$
    there is a natural way of defining seminorm $s_\alpha(R)=\|R_\alpha\|_{\mathcal{H}_\alpha},$ for every $\alpha \in \Lambda.$ This gives an upward filtered family $\{s_\alpha\}_{\alpha \in \Lambda}$ of $C^{\ast}$-seminorms that generate complete Hausdorff locally convex topology on $\mathcal{A}[T].$ If we define $\mathcal{I}_\alpha=s^{-1}_\alpha(0),$ then $\mathcal{I}_\alpha$ is a closed two sided $\ast$-ideal in $\mathcal{A}[T]$ and $\mathcal{A}[T]/ \mathcal{I}_{\alpha}$ is a $C^{\ast}$-algebra with the norm induced by $s_{\alpha}$, for every $\alpha \in \Lambda.$ Now for a fixed $\alpha \in \Lambda,$ we define $\theta_\alpha \colon \mathcal{A}[T]/\mathcal{I}_\alpha \longrightarrow C^{\ast}\{I_{\mathcal{H}_{\alpha}},\;T_{\alpha},T^{\ast}_\alpha\}$ by
\begin{equation}\label{eq: theta_alpha}
    \theta_\alpha(R+\mathcal{I}_\alpha)=R_\alpha, \text{ for all }R+\mathcal{I}_\alpha \in \mathcal{A}[T]/\mathcal{I}_\alpha.
\end{equation}
If $R + \mathcal{I}_{\alpha}= R^{\prime} + \mathcal{I}_{\alpha} $ then $R - R^{\prime} \in \mathcal{I}_{\alpha}$ and so $R_{\alpha} = R^{\prime}_{\alpha}.$ This shows that $\theta_{\alpha}$ is a well-defined linear map. Now we show that $\theta_{\alpha}$ is $C^{\ast}$-isomorphism. Let $R+\mathcal{I}_\alpha,\; R^{\prime} + \mathcal{I}_{\alpha}  \in \mathcal{A}[T]/ \mathcal{I}_{\alpha}$. Then 
\begin{equation*}
    \theta_{\alpha}\left( R R^{\prime}+ \mathcal{I}_{\alpha}\right) = R_{\alpha} R_{\alpha}^{\prime} \; \text{and}\; \theta_{\alpha}\left( R^{\ast}+\mathcal{I}_\alpha \right)= \theta_{\alpha}\left( R+\mathcal{I}_\alpha \right)^{\ast}.
\end{equation*}
Moreover, 
\begin{align*}
    \left\| \theta_{\alpha}\left( R+\mathcal{I}_\alpha \right)\right\|_{\mathcal{H}_{\alpha}} =  \|R_{\alpha}\|_{\mathcal{H}_{\alpha}} = s_{\alpha} \left( R \right).
\end{align*}
It follows that $\theta_\alpha$ is injective. Also $\theta_\alpha$ is surjective because given any $R_\alpha \in C^{\ast}\{I_{\mathcal{H}_{\alpha}},\;T_{\alpha},T^{\ast}_\alpha\},$ there exists an $R_\beta \in C^{\ast}\{I_{\mathcal{H}_{\beta}},\;T_{\beta},T^{\ast}_\beta\}$ such that $\phi_{\alpha,\beta}(R_\beta)=R_\alpha,$ whenever $\alpha \leq \beta$ and $\varprojlim\limits_{\beta \in \Lambda}R_\beta+\mathcal{I}_\alpha$ is the desired pre-image of $R_\alpha$ under the map $\theta_\alpha.$ 
Therefore, $\theta_{\alpha}$ is an isometric $\ast$-isomorphism of $C^{\ast}$-algebras. As a consequence, we have that maximal ideal spaces $\mathcal{M}_{\mathcal{A}[T]/ \mathcal{I}_{\alpha}}$ and $\mathcal{M}_{C^{\ast}\{I_{\mathcal{H}_{\alpha}}, T_{\alpha},  T_{\alpha}^{\ast}\}}$ are homeomorphic.  Since $T_{\alpha} \in \mathcal{B}(\mathcal{H}_{\alpha})$ is normal, it follows from the spectral theorem \cite[Theorem 10.2]{kehezhu} that $\mathcal{M}_{C^{\ast}\{I_{\mathcal{H}_{\alpha}}, T_{\alpha},  T_{\alpha}^{\ast}\}}$ equipped with the weak$^{\ast}$-topology is homeomorphic to the spectrum $\sigma(T_{\alpha})$ for each $\alpha \in \Lambda. $ Finally, $\mathcal{M}_{\mathcal{A}[T]/ \mathcal{I}_{\alpha}}$ is homeomorphic to $\sigma(T_{\alpha})$ for each $\alpha \in \Lambda.$ 

On the other hand, from Definition \ref{defn: characterspace}, the character space of the commutative unital locally $C^{\ast}$-algebra generated by the locally bounded normal operator $T$ is given by 
\begin{equation*}
    \mathcal{M}_{\mathcal{A}[T]} = \bigcup\limits_{\alpha \in \Lambda}\Big\{ \varphi_\alpha \circ {\pi}_\alpha \colon \varphi_\alpha \in \mathcal{M}_{\mathcal{A}[T]/\mathcal{I}_\alpha}\Big\}.
\end{equation*}
It is equipped with the inductive limit topology as described in Section \ref{Section 2}. Note that, whenever $\alpha \leq \beta$ we see that $\sigma(T_{\alpha})\subseteq \sigma(T_{\beta})$ and the family $\{\sigma(T_{\alpha})\}_{\alpha \in \Lambda}$ forms a strictly inductive system of compact topological spaces in $\mathbb{C}$. Hence the inductive limit given by
\begin{equation*}
    \sigma_{loc}(T) = \bigcup\limits_{\alpha \in \Lambda} \sigma(T_{\alpha}),
\end{equation*}
is equipped with the inductive limit topology. By using $(2)$ of Remark \ref{rmk: continuity and homeomorphism} and the above various homeomorphism relations, we conclude that there is a homeomorphism between $\mathcal{M}_{\mathcal{A}[T]}$ and $\sigma_{loc}(T).$ In fact, the map $\Delta \colon \mathcal{M}_{\mathcal{A}[T]} \to \sigma_{loc}(T)$ defined by 
\begin{equation*}
    \Delta \left( \varphi_{\alpha} \circ \pi_{\alpha} \right) = \left(\varphi_{\alpha} \circ \theta_{\alpha}^{-1}\right)(T_{\alpha}),
\end{equation*}
for any $\varphi_{\alpha} \in \mathcal{M}_{\mathcal{A}[T]/ \mathcal{I}_{\alpha}},\; \alpha \in \Lambda$ is a homeomorphism.
\end{proof}
 Next we turn our discussion towards the functional calculus of locally bounded normal operator $T$. Let us recall the Gelfand representation in this context. It follows from \cite[Theorem 10.2]{kehezhu} that, for each $\alpha \in \Lambda,$ the Gelfand representation $\Gamma_\alpha \colon C^{\ast}\{I_{\mathcal{H}_{\alpha}},\;T_{\alpha},T^{\ast}_\alpha\} \rightarrow C\big(\sigma(T_\alpha)\big)$ is defined such that 
\begin{equation*}
    \Gamma_\alpha\big(\mathcal{P}(T_\alpha,T^{\ast}_\alpha)\big)(\lambda)=\mathcal{P}(\lambda,\bar{\lambda}),\text{ for every }\mathcal{P}(T_\alpha,T^{\ast}_\alpha) \in C^{\ast}\{I_{\mathcal{H}_{\alpha}},\;T_{\alpha},T^{\ast}_\alpha\},\; \lambda \in \sigma(T_\alpha).
\end{equation*}
Since the class of polynomials $p(T_\alpha,T^{\ast}_{\alpha})$ is dense in $C^{\ast}\{I_{\mathcal{H}_{\alpha}},\;T_{\alpha},T^{\ast}_\alpha\},$ then $\Gamma_{\alpha}$ has a unique continuous extension and it is again denoted by $\Gamma_\alpha$.
Note that each $\sigma(T_{\alpha})$ is a compact Hausdorff space and $\sigma(T_{\alpha}) \subseteq \sigma(T_{\beta})$ since $T_{\beta}\big|_{\mathcal{H}_{\alpha}} = T_{\alpha}$ whenever $\alpha \leq \beta.$ As a result, $\sigma_{loc}(T) = \bigcup\limits_{\alpha}{\sigma(T_{\alpha})}$ can be seen as an (strictly) inductive limit of the strictly inductive family $\{\sigma(T_{\alpha})\}_{\alpha \in \Lambda}$ of toplogical space. In fact, by Proposition \ref{prop: topological properties of character sp} we see that $\sigma_{loc}(T)$ is completely regular. Similar to the result proved in Theorem \ref{Thm: C(M_A) is local C* alg}, $C\big(\sigma_{loc}(T)\big)$ is a commutative unital locally $C^{\ast}$-algebra. A continuous function $f \in C\big(\sigma_{loc}(T)\big)$ is uniquely represented by its restrictions $f\big|_{\sigma(T_{\alpha})}$, for $\alpha \in \Lambda.$
\begin{thm}\label{Thm: proj system loc bdd operator}
Let $T \in C^\ast_{\mathcal{E}}(\mathcal{D})$ is normal and $f \in C\big(\sigma_{loc}(T)\big).$ For every $\alpha \in \Lambda$, define  $f_\alpha:=f\big|_{\sigma(T_\alpha)}$ then the family $\Big\{\Gamma^{-1}_{\alpha}(f_\alpha)\Big\}_{\alpha\in \Lambda}$ of bounded operators satisfy that $\mathcal{H}_{\alpha}$ is a reducing subspace of $\Gamma^{-1}_\beta(f_\beta)$ and  
\begin{equation*}
    \Gamma^{-1}_\beta(f_\beta)\big|_{\mathcal{H}_\alpha}=\Gamma^{-1}_\alpha(f_\alpha), \text{ whenever } \alpha \leq \beta.
\end{equation*}

\end{thm}
\begin{proof}
Since $T \in C^\ast_{\mathcal{E}}(\mathcal{D})$ is normal, for each $\alpha \in \Lambda$, the operator $T_\alpha=T\big|_{\mathcal{H}_\alpha} \in \mathcal{B}(\mathcal{H}_\alpha)$ is normal. Therefore, for every $f  \in C(\sigma_{loc}(T)),$ the mapping $f_\alpha \mapsto \Gamma^{-1}_{\alpha}(f_\alpha)=f_\alpha(T_\alpha)$ from $C\big(\sigma(T_\alpha)\big)$ onto $C^{\ast}\{I_{\mathcal{H}_{\alpha}},\;T_{\alpha},T^{\ast}_\alpha\}$ is the continuous functional calculus for the normal operator $T_\alpha.$ 

Suppose $\alpha \leq \beta$ and Since $\sigma(T_{\beta})$ is compact, for $f_{\beta} \in C\left( \sigma(T_{\beta})\right)$ there is a sequence $\{\mathcal{P}_{n}\}_{n\in \mathbb{N}}$ of polynomials converges uniformly to $f_{\beta}$ on $\sigma(T_{\beta})$. Moreover, $\{\mathcal{P}_{n}\}_{n\in \mathbb{N}}$ converges uniformly to $f_{\alpha}$ on $\sigma(T_{\alpha})$. If $Q_{\alpha, \beta}\colon \mathcal{H}_{\beta} \to \mathcal{H}_{\beta}$ is an orthogonal projection onto $\mathcal{H}_{\alpha}$, then by using the fact that $Q_{\alpha, \beta}T_{\beta} = T_{\beta} Q_{\alpha, \beta}$, for any $x \in \mathcal{H}_{\beta}$, we have 
\begin{align*}
    Q_{\alpha, \beta} \Gamma_{\beta}^{-1}\big( f_{\beta}\big) x = Q_{\alpha, \beta} \Big( \lim\limits_{n \to \infty} \Gamma_{\beta}^{-1}\left( \mathcal{P}_{n}\right)\Big) x &= Q_{\alpha, \beta} \Big( \lim\limits_{n \to \infty}  \mathcal{P}_{n}\left( T_{\beta}, T_{\beta}^{\ast}\right)\Big) x \\
    &= \lim\limits_{n \to \infty} \Big( Q_{\alpha, \beta}  \mathcal{P}_{n}\left( T_{\beta}, T_{\beta}^{\ast}\right)  \Big)x\\
    &=\lim\limits_{n \to \infty}    \mathcal{P}_{n}\left( T_{\beta}, T_{\beta}^{\ast} \right) Q_{\alpha, \beta} x\\
    &= \Gamma_{\beta}^{-1}\big( f_{\beta}\big) Q_{\alpha, \beta} x.
\end{align*}
This shows that $\mathcal{H}_{\alpha}$ is a reducing subspace of $\Gamma_{\beta}^{-1}\big( f_{\beta}\big).$ Since $T_{\beta}\big|_{\mathcal{H}_{\alpha}} = T_{\alpha}$,  for every $x \in \mathcal{H}_{\alpha}$  we get 
\begin{align*}
    \Gamma_{\beta}^{-1}\big( f_{\beta}\big) x = \lim\limits_{n \to \infty} \Gamma_{\beta}^{-1}\left( \mathcal{P}_{n}\right)x =  \lim\limits_{n \to \infty} \mathcal{P}_{n}\left( T_{\beta}, T_{\beta}^{\ast}\right) x &= \lim\limits_{n \to \infty} \mathcal{P}_{n}\left( T_{\alpha}, T_{\alpha}^{\ast}\right) x\\
    &= \lim\limits_{n \to \infty} \Gamma_{\alpha}^{-1}\left( \mathcal{P}_{n}\right)x \\
    &=  \Gamma_{\alpha}^{-1}\big( f_{\alpha}\big) x.
\end{align*}
Therefore, $\Gamma_{\beta}^{-1}\big( f_{\beta}\big)\big|_{\mathcal{H}_{\alpha}} = \Gamma_{\alpha}^{-1}\big( f_{\alpha}\big),$ whenever $\alpha \leq \beta.$ Consequently, we see that the family $\left\{\Gamma^{-1}_{\alpha}(f_\alpha)\right\}_{\alpha \in \Lambda}$ is in $\varprojlim\limits_{\alpha \in \Lambda}C^{\ast}\{I_{\mathcal{H}_{\alpha}}, T_{\alpha},  T_{\alpha}^{\ast}\}.$
\end{proof}
 
\begin{rmk}
    In view of Theorem \ref{Thm: proj system loc bdd operator} and Equation \eqref{eq: Gelfand type representation}, for a locally bounded normal operator $T \in C^{\ast}_{\mathcal{E}}(\mathcal{D}),$
the Gelfand type representation $\Gamma \colon \mathcal{A}[T] \rightarrow C\big(\sigma_{loc}(T)\big)$ is given by
\begin{equation*}
    \Gamma(R)\big|_{\sigma(T_\alpha)}=\Gamma_\alpha(R_\alpha), \text{ for all }R=\varprojlim\limits_{\beta \in \Lambda}R_\beta \in \mathcal{A}[T].
\end{equation*}
Clearly, the map $\Gamma$ is well-defined since if $R=R^{\prime},$ then $\Gamma_{\alpha}(R_\alpha)=\Gamma_{\alpha}(R_{\alpha}^{\prime}),$ for all $\alpha \in \Lambda.$
\end{rmk}
\begin{defn}\label{defn: cont func calculus}
    For a given $f \in C(\sigma_{loc}(T))$, by using Theorem \ref{Thm: proj system loc bdd operator} we define the \textit{continuous functional calculus} of $f$ at $T$ as the locally bounded operator $f(T)$ given by is the projective limit of the family $\left\{\Gamma^{-1}_{\alpha}(f_\alpha)\right\}_{\alpha \in \Lambda}$. That is,  
\begin{equation}\label{eq: functional calculus}
    f(T):=\displaystyle\varprojlim_{\alpha \in \Lambda}\Gamma^{-1}_{\alpha}(f_\alpha).
\end{equation} 
\end{defn}
\begin{thm}\label{Thm: cont functional calculus properties}
Let $T \in C^{\ast}_{\mathcal{E}}(\mathcal{D})$ be a  normal operator. Then the map $\Phi \colon C\big(\sigma_{loc}(T)\big) \rightarrow \mathcal{A}[T]$ defined by 
\begin{equation*}
    \Phi(f) = f(T) = \displaystyle\varprojlim_{\alpha \in \Lambda}\Gamma^{-1}_{\alpha}(f_\alpha),
\end{equation*}
is a coherent local contractive (isometric) $\ast$-homomorphism. Moreover, 
\begin{enumerate}
    \item If $\mathcal{P}(z,\bar{z})$ is a polynomial of two variables, then $\mathcal{P}(T) = \varprojlim\limits_{\alpha \in \Lambda}\mathcal{P}(T_\alpha,T^{\ast}_{\alpha}).$ In particular, if $f(z) = z$ is the identity function on $C\big(\sigma_{loc}(T))\big)$, then $f(T) = T.$
    \item $T$ is self adjoint if and only if $\sigma_{loc}(T) \subseteq \mathbb{R}.$
    \item $T$ is unitary if and only if $\sigma_{loc}(T) \subseteq \partial\mathbb{D},$ where $\mathbb{D}=\{z \in \mathbb{C} \colon |z|<1\}.$
\end{enumerate}
    
\end{thm}

\begin{proof} Since $\sigma_{loc}(T)$ is a (strict) inductive limit of strictly inductive system $\{\sigma(T_{\alpha})\}_{\alpha \in \Lambda}$ of compact sets, one can see that $C\left(\sigma_{loc}(T)\right)$ is the projective limit of projective system $\big\{C(\sigma(T_{\alpha}))\big\}_{\alpha \in \Lambda}$ of commutative unital $C^{\ast}$-algebras.  In fact, $f \mapsto \| f_{\alpha}\|_{\infty}:= \sup\limits\Big\{ |f_{\alpha}(t)|:\; t\in \sigma(T_{\alpha})\Big\}$ defines an upward filtered family of $C^{\ast}$- semi-norm on $C\left(\sigma_{loc}(T)\right)$. For each $\alpha \in \Lambda,$ the map $\zeta_{\alpha} \colon C\left(\sigma_{loc}(T)\right) \to C(\sigma(T_{\alpha}))$ defined by $\zeta_\alpha(f)=f_\alpha=f\big|_{\sigma(T_\alpha)}$ for all $f \in C\big(\sigma_{loc}(T)\big)$ is a $\ast$-homomorphism. 

Let $f,g \in C(\sigma_{loc}(T))$ and $\lambda \in \mathbb{C}$. Then 
\begin{align*}
    \Phi\left( (f+\lambda g)^{\ast}\right) = (f+\lambda g)^{\ast}(T)&=\varprojlim\limits_{\alpha \in \Lambda}\Gamma^{-1}_{\alpha}\big(f^{\ast}_{\alpha}+\bar{\lambda} g_{\alpha}^{\ast}\big) \\
    &=\varprojlim\limits_{\alpha \in \Lambda}\Gamma^{-1}_\alpha(f_\alpha)^{\ast}+\bar{\lambda}\varprojlim\limits_{\alpha \in \Lambda}\Gamma^{-1}_\alpha(g_\alpha)^{\ast} \\
    &=\Phi(f)^{\ast}+\bar{\lambda}\Phi(g)^{\ast}
\end{align*}
Therefore $\Phi$ is a $\ast$-homomorphism.
 As shown in Theorem \ref{Thm: proj system loc bdd operator}, there is a family $\Big\{\Gamma_{\alpha}^{-1}\colon C\big(\sigma(T_\alpha)\big) \rightarrow C^{\ast}\{I_{\mathcal{H}_{\alpha}}, T_{\alpha},  T_{\alpha}^{\ast}\}\Big\}_{\alpha \in \Lambda}$ of $C^{\ast}$-representations satisfying,
\begin{align*}
    \left(\phi_\alpha \circ \Phi\right) (f)=\phi_\alpha\big(f(T)\big)= \phi_\alpha\left(\displaystyle\varprojlim_{\beta \in \Lambda}\Gamma^{-1}_{\beta}(f_\beta)\right) 
    = \Gamma^{-1}_\alpha(f_\alpha) 
    = \Gamma^{-1}_\alpha\big(f\big|_{\sigma(T_\alpha)}\big)=\left(\Gamma^{-1}_\alpha\circ \zeta_\alpha\right)(f),  
\end{align*}
for every $f \in C\big(\sigma_{loc}(T)\big).$ Equivalently, we have the following commuting diagram:
\begin{equation*}
    \begin{tikzcd}[every matrix/.append style={name=m},
    execute at end picture={
    \draw [<-] ([xshift=4mm,yshift=3mm]m-2-2.west) arc[start angle=-90,delta angle=270,radius=0.33cm]; }] 
    C\big(\sigma_{loc}(T)\big)  \arrow[thick,swap] {dd}{\zeta_\alpha} \arrow[thick]{rr}{\Phi} && \mathcal{A}[T]  \arrow[thick]{dd}{\phi_\alpha} \\
     && \\
    C\big(\sigma(T_\alpha)\big) \arrow[thick,swap]{rr}{\Gamma^{-1}_\alpha} && C^{\ast}\{I_{\mathcal{H}_{\alpha}},\;T_{\alpha},T^{\ast}_\alpha\}  
    \end{tikzcd}
\end{equation*}

This shows that $\Phi$ is a coherent representation of locally $C^{\ast}$-algebras. Further, we see that the map $\Phi$ is a local isometric representation in the sense that for each $\alpha$, we have
\begin{align*}
    s_\alpha(\Phi(f))=s_\alpha(f(T))=\|\Gamma^{-1}_{\alpha}(f_\alpha)\|_{\mathcal{H}_\alpha}=\|f_\alpha\|_{\infty}, \text{ for all }f \in C\big(\sigma_{loc}(T)\big).
\end{align*}
Hence $\Phi$ is a coherent local contractive (isometric) $\ast$-homomorphism.

\noindent Proof of $(1):$ If $f(z)=p(z,\bar{z}),$ then $\Gamma^{-1}_\alpha(T_\alpha)=f_\alpha(T_\alpha)=p(T_\alpha,T^{\ast}_{\alpha}).$ Consequently, $f(T)=\varprojlim\limits_{\alpha \in \Lambda}p(T_\alpha,T^{\ast}_{\alpha}).$ In particular, if $f(z)=z,$ then $f(T)=\varprojlim\limits_{\alpha \in \Lambda}T_\alpha=T.$ 

\noindent Proof of $(2)$: Since $T=T^{\ast},$ for each $\alpha,\; T_\alpha$ is a bounded self adjoint operator. So by \cite[Theorem 10.4(a)]{kehezhu} $\sigma(T_\alpha) \subseteq \mathbb{R},$ for all $\alpha \in \Lambda.$ Therefore by Theorem \ref{Theorem: local spectrum}, we have $\sigma_{loc}(T)=\bigcup\limits_{\alpha \in \Lambda}\sigma(T_\alpha) \subseteq \mathbb{R}.$ On the other hand, if $\sigma_{loc}(T) \subseteq \mathbb{R},$ then $T^{\ast}_\alpha=T_\alpha,$ for every $\alpha \in \Lambda$ and hence $T=T^{\ast}.$

\noindent Proof of $(3):$ For every $\alpha \in \Lambda,\;T_\alpha$ is unitary if and only if $\sigma(T_\alpha) \subseteq \partial\mathbb{D}.$ Finally the result follows from Theorem \ref{Theorem: local spectrum}.
\end{proof}
The following example, motivated from \cite[Example 3.1]{dosiev}, gives a description of the notion of continuous functional calculus and illustrates a particular example of Spectral mapping theorem in this set up, which will be proved in Theorem \ref{Thm: loc. spectral mapping thm}.
\begin{eg} \label{eg: functional calculus}
    Let us consider the directed POSET $(\mathbb{N}, \leq)$ and the Hilbert space $\mathcal{H}=L^2(\mathbb{R}).$ For each $n \in \mathbb{N},$ define
    \begin{equation*}
        \mathcal{H}_n := \Big\{h \in L^2(\mathbb{R}):\;  supp(h) \subseteq [-n,n]\Big\},
    \end{equation*}
    which is a closed subspace of $\mathcal{H}$ and hence a Hilbert space. Further, if $m \leq n$ then $\mathcal{H}_{m} \subseteq \mathcal{H}_{n}.$ This implies that $\mathcal{E}:= \big\{ \mathcal{H}_{n}\big\}_{n \in \mathbb{N}}$ is an upward filtered family of Hilbert spaces. Therefore, $\mathcal{D}=\displaystyle\bigcup_{n\in \mathbb{N}}\mathcal{H}_n$ is a locally Hilbert space and $\overline{\mathcal{D}}=\mathcal{H}.$ Now define $T \colon \mathcal{D} \rightarrow \mathcal{D}$ by
    \begin{equation*}
        T(h)(x)=x\, h(x), \text{ for all } h \in \mathcal{D},\, x \in supp(h).
    \end{equation*}
    In particular, for each $n \in \mathbb{N}$, it follows that $T_n=T\big|_{\mathcal{H}_n} \in \mathcal{B}(\mathcal{H}_{n})$ is normal with $\sigma(T_n)=[-n,n]$. As a result, $T \in C^{\ast}_\mathcal{E}(\mathcal{D})$ is normal. The local spectrum of $T$ (see Equation \ref{eq: local spectrum}) is given by   
    \begin{equation*}
        \sigma_{loc}(T)=\bigcup\limits_{n \in \mathbb{N}}\sigma(T_n)=\bigcup\limits_{n \in \mathbb{N}}[-n,n]=\mathbb{R}.
    \end{equation*}
    Let $f \in C(\mathbb{R}).$ For each $n \in \mathbb{N}$, $f_{n} \in C\left( [-n, n]\right)$ and the bounded operator $f_{n}(T_{n})$ is given by the functional calculus of $T_{n}.$  Following Theorem \ref{Thm: proj system loc bdd operator} Equation \eqref{eq: functional calculus} we have  $f(T)=\varprojlim\limits_{n \in \mathbb{N}}f_n(T_{n})$.  For instance, consider $f(x)=e^x,\;x \in \mathbb{R}.$ If $h \in \mathcal{D}$ then $supp(h) \subseteq [-n, n]$ for some $n \in \mathbb{N}$ and 
    \begin{equation*}
        f(T)(h)(x) = e^x h(x) \; \text{for all}\; x \in [-n, n].
    \end{equation*}
    Further, the local spectrum of $f(T)$ is computed as,
    \begin{align*}
        \sigma_{loc}\big(f(T)\big)=\bigcup\limits_{n \in \mathbb{N}}\sigma\big(f_{n}(T_{n})\big)=\bigcup\limits_{n \in \mathbb{N}}\big\{e^{x} :\; x \in [-n,n]\big\}=\big\{e^{x} \; : x \in \mathbb{R}\big\}.
    \end{align*}
    It follows that $ \sigma_{loc}\big(f(T)\big)=f\big(\sigma_{loc}(T)\big).$ 
\end{eg}
In the following theorem, our aim is to generalize the last observation we made in Example \ref{eg: functional calculus} to a locally bounded normal operator $T$. We call this result as {\it the local spectral mapping theorem} in this context as it gives the relation between the local spectrum of $f(T)$ and the range of $f$ on the local spectrum of $T,$ for any $f \in C\big(\sigma_{loc}(T)\big).$ 
For the spectral mapping theorem for unbounded normal operator, we refer the reader \cite[Theorem 6.3]{paliogiannis}. However, our result can be seen as a refiend one to the result given in \cite[Theorem 6.3]{paliogiannis} in the special case of locally bounded operators (not necessarily a closed operator). 
\begin{thm} \label{Thm: loc. spectral mapping thm}(The local spectral mapping theorem)
Let $T \in C^{\ast}_{\mathcal{E}}(\mathcal{D})$ and $f \in C(\sigma_{loc}(T)).$ Then 
\begin{equation*}
        \sigma_{loc}(f(T))=f(\sigma_{loc}(T))=\big\{f(\lambda)\colon \lambda \in \sigma_{loc}(T)\big\}.
\end{equation*}
\end{thm}
\begin{proof}
   If we denote $T_\alpha=T\big|_{\mathcal{H}_\alpha}$ then $T_{\alpha} \in \mathcal{B}(\mathcal{H}_{\alpha})$ is normal and  $\sigma(T_{\alpha})$ is compact for each $\alpha \in \Lambda$. Since $f_{\alpha} \in C(\sigma(T_{\alpha}))$, by the spectral mapping theorem \cite[Theorem 10.3(c)]{kehezhu} of bounded normal operators, we have 
   \begin{equation*}
       \sigma\left(f_{\alpha}(T_{\alpha})\right) = f_{\alpha}\left(\sigma(T_{\alpha})\right),
   \end{equation*}
   for every $\alpha \in \Lambda.$ Following Equation \eqref{eq: functional calculus} and Theorem \ref{Theorem: local spectrum}, we get
    \begin{align*}
    \sigma_{loc}(f(T))=\bigcup\limits_{\alpha\in \Lambda}\sigma\big(f_\alpha(T_\alpha)\big)= \bigcup\limits_{\alpha\in \Lambda} f_{\alpha} \left( \sigma(T_{\alpha})\right)  &= \bigcup\limits_{\alpha\in \Lambda} f \left( \sigma(T_{\alpha})\right) \\
    &= f \big( \bigcup\limits_{\alpha\in \Lambda}   \sigma(T_{\alpha})\big)\\
    &= f\left( \sigma_{loc}(T)\right)\\
    &= \Big\{ f(\lambda) \colon \lambda \in \sigma_{loc}(T)\Big\}.
\end{align*}
Therefore, $f\big(\sigma_{loc}(T)\big)= \sigma_{loc}\big(f(T)\big).$
\end{proof}
\subsection*{Declaration} 
The authors declare that there are no conflicts of interest.
\section*{Acknowledgement}
The first named author would like to thank SERB (India) for a financial support in the form of Startup Research Grant (File No. SRG/2022/001795). The authors express their sincere thank to DST in the form of the FIST grant (File No. SR/FST/MS-I/2019/46(C)) and IISER Mohali for providing necessary facilities to carryout this work. 

\bibliographystyle{plain}

\begin{thebibliography}{99}
\bibitem{AllanG} G.~R. Allan, On a class of locally convex algebras, Proc. London Math. Soc. (3) {\bf 17} (1967), 91--114.

     \bibitem{arens} R.~F. Arens, A generalization of normed rings, Pacific J. Math. {\bf 2} (1952), 455--471.

     \bibitem{SanthoshKP} B.~V.~R. Bhat, A. Ghatak and P. Santhosh~Kumar, Stinespring's theorem for unbounded operator valued local completely positive maps and its applications, Indag. Math. (N.S.) {\bf 32} (2021), no.~2, 547--578.

     \bibitem{dosiev} Anar ~A. Dosiev, Local operator spaces, unbounded operators and multinormed $C^*$-algebras, J. Funct. Anal. {\bf 255} (2008), no.~7, 1724--1760.

     \bibitem{gheondea} A. Gheondea, Operator models for Hilbert locally $C^*$-modules, Oper. Matrices {\bf 11} (2017), no.~3, 639--667.

     \bibitem{Inoue} A. Inoue, Locally $C\sp{\ast} $-algebra, Mem. Fac. Sci. Kyushu Univ. Ser. A {\bf 25} (1971), 197--235.

     \bibitem{Joita} M. Joi\c ta, Locally von Neumann algebras, Bull. Math. Soc. Sci. Math. Roumanie (N.S.) {\bf 42(90)} (1999), no.~1, 51--64.

     \bibitem{kothe} G. K\"othe, {\it Topological vector spaces. I}, translated from the German by D. J. H. Garling, 
Die Grundlehren der mathematischen Wissenschaften, Band 159, Springer-Verlag New York, Inc., New York, 1969.

    \bibitem{michael} E.~A. Michael, Locally multiplicatively-convex topological algebras, Mem. Amer. Math. Soc. {\bf 11} (1952), 79 pp.

    \bibitem{narici} L. Narici and E. Beckenstein, {\it Topological vector spaces}, second edition, 
Pure and Applied Mathematics (Boca Raton), 296, CRC Press, Boca Raton, FL, 2011.

    \bibitem{George} G.~F. Nassopoulos, Spectral decomposition and duality in commutative locally $C^*$-algebras, in {\it Topological algebras and applications}, 303--317, Contemp. Math., 427 (2007), Amer. Math. Soc., Providence, RI,

    \bibitem{paliogiannis} F.~C. Paliogiannis, Some remarks on the spectral theory and commutativity of unbounded normal operators, Complex Anal. Oper. Theory {\bf 8} (2014), no.~3, 733--744.

    \bibitem{Phillips} N.~C. Phillips, Inverse limits of $C^*$-algebras and applications, in {\it Operator algebras and applications, Vol.\ 1}, 127--185, London Math. Soc. Lecture Note Ser., 135, Cambridge Univ. Press, Cambridge, .

    \bibitem{schmudgen} K. Schm\"udgen, \"Uber LMC-Algebren, Math. Nachr. {\bf 68} (1975), 167--182.

    \bibitem{voiculescu} D.~V. Voiculescu, Dual algebraic structures on operator algebras related to free products, J. Operator Theory {\bf 17} (1987), no.~1, 85--98.

    \bibitem{kehezhu} K. Zhu, {\it An introduction to operator algebras}, Studies in Advanced Mathematics, CRC, Boca Raton, FL, 1993.
 \end{thebibliography}

\end{document}